\documentclass[11pt, reqno]{amsart}
\usepackage{mathtools, amsmath, amssymb, amsthm, breqn, mathrsfs, hyperref, latexsym, amsfonts, amsrefs, cleveref}
\usepackage{thmtools}
\usepackage[mathscr]{eucal}

\numberwithin{equation}{section}

\theoremstyle{definition}
\newtheorem{definition}{Definition}[section]
\newtheorem{example}[definition]{Example}

\newtheorem{remark}[definition]{Remark}

\theoremstyle{plain}
\newtheorem{theorem}[definition]{Theorem}
\newtheorem{lemma}[definition]{Lemma}
\newtheorem{proposition}[definition]{Proposition}
\newtheorem{corollary}[definition]{Corollary}

\newtheorem{result}[definition]{Result}

\newcommand\hull[1]{\widehat{#1}}

\newcommand{\rl}{\Re\mathfrak{e}}

\newcommand{\id}{\mathbb{I}}

\newcommand{\smoo}{\mathcal{C}}
\newcommand{\holo}{\mathcal{O}}

\newcommand{\cplx}{\mathbb{C}}

\newcommand{\rea}{\mathbb{R}}

\newcommand{\C}{\mathbb{C}}
\newcommand{\R}{\mathbb{R}}
\renewcommand{\L}{\mathscr{L}}

\newcommand{\ehat}{\widehat{\mkern6mu}}

\begin{document}

\title [Carleman Approximation]{Carleman Approximation for certain sets with an isolated singularity}
\author{Harshith Alagandala and Sushil Gorai}
	\address{Department of Mathematics, Western University,
		London, ON, Canada -- N6A 3K7}
	\email{halagand@uwo.ca, sairaj922@gmail.com}
	
	\address{Department of Mathematics and Statistics, Indian Institute of Science Education and Research Kolkata,
		Mohanpur -- 741 246}
	\email{sushil.gorai@iiserkol.ac.in, sushil.gorai@gmail.com}
	\thanks{}

	\date{\today}

\maketitle

\begin{abstract}
    In this paper, we prove that local polynomial convexity at the origin for 
		the union of finitely many transverse totally real subspaces  of maximal dimension
		is sufficient for Carleman approximation. 
		Some new conditions are given for the polynomial convexity of the union of three transverse totally real planes in $\cplx^2$.
		We also provide a sufficient condition on the union of two Lipschitz graphs for Carleman approximation. 
    Along the way, we provide sufficient conditions for union of two Lipschitz graphs to be polynomially convex.
		Finally, we find a family of surfaces in $\C^2$ with a hyperbolic complex point that allows Carleman approximation.
\end{abstract}

\section{Introduction and the statements of the results}
\label{sec:intro}

A closed subset  $S$  of $\C^n$  is said to be a Carleman set if
for $f \in \smoo(S, \C)$ and
a strictly positive real valued continuous function $\epsilon$,
there exists $g \in \holo(\C^n)$ such that
\[
  \left| f(z) - g(z) \right| < \epsilon(z)\;\; \forall z \in S.
\]
We say that $S$ allows Carleman approximation if $S$ is a Carleman set.
In 1927, Carleman\cite{carleman1927} showed that the set $\mathbb{R} \subset \C$ allows Carleman approximation.
In one variable, Carleman sets are well understood.
A set $S\subset \cplx$ is a Carleman set if and only if
it is polynomially convex, locally connected at infinity and $\text{Int} S=\varnothing$ \cite[Remark 2.8]{manne-wold-ovrelid2011}.
In higher dimensions, totally real affine subspaces of $\C^n$
are Carleman sets (Scheinberg \cite{scheinberg76}). 
A necessary condition for Carleman approximation is polynomial convexity.
A compact subset $K$ is called polynomially convex if 
$$K=\hull{K}:=\{z\in\C^n: |p(z)|\leq\sup_K|p|\;\;\forall p\in\C[z_1,\dots,z_n]\}.$$
A closed subset $M \subset \C^n$ is called polynomially convex
if there is a sequence  $\{K_m\}$ of polynomially convex compact subsets such that
$K_m \subset K_{m+1}$, $\cup K_m = M$.
Magnusson and Wold \cite{magnusson2016characterization} proved that
a stratified totally real set allows Carleman approximation if and
only if it is polynomially convex and has bounded $E$-hulls. 
\begin{itemize}
\item A closed set $M\subset \C^n$ is said to have bounded $E$-hull
    if for any compact  set $K \subset \C^n$,
    the set $\widehat{K \cup M} \setminus (K\cup M)$ is bounded.

\item  A subset of $X \subset \C^n$ is said to be a stratified totally real set
    if $X$ is the increasing union $X_0\subset X_1\subset\dots\subset X_n=X$ of closed sets,
    such that $X_j\setminus X_{j-1}$ is totally real set.

\end{itemize}
For a given closed set, it is quite challenging to show 
whether it is polynomially convex or has bounded E-hull. 
In this paper,
we consider three different setting for proving Carleman approximation directly: 

\begin{itemize}
    \item [(i)] Finite union of mutually transverse maximal totally real subspaces in $\cplx^n$ which is locally polynomially convex at the origin.
    \smallskip
    
    \item [(ii)] Union of two transverse Lipschitz graphs in $\cplx^2$.
    \smallskip

    \item [(iii)] Real surfaces in $\cplx^2$ with a nondegenerate hyperbolic complex point.
\end{itemize}
We now describe both of them in a bit more detail.

\subsection{The union of finitely many totally real subspaces}
Let  $L_1, ... , L_m$ be maximal totally real subspaces such that
$L_k \cap L_j = \{0\}$ for all $k\neq j$.
Assume that $\bigcup_{j=1}^m L_j$ is polynomially convex. 
One of the aim of this article is to show that $\bigcup_{j=1}^m L_j$ allows Carleman approximation by entire functions in $\C^n$.
Clearly, $\bigcup_{j=1}^m L_j$ is a stratified totally real set.
It is not clear and very difficult to determine if this union has bounded $E$-hull. 
For the union of two totally real maximal subspaces, the answer was given by Manne \cite{manne1994carlemantwo}:
The union of two maximal totally real subspaces $L_1$ and $L_2$ which
intersect only at the origin allows Carleman approximation
when $L_1 \cup L_2$ is polynomially convex.

Assume that $L_j$ for $j=0,\dots, m$ are maximal totally real subspaces of $\C^n$ with $L_k \cap L_j = \{0\}$ for $k\neq j$.
After applying an invertible $\C$-linear map, we can identify
\[L_0=\rea^n\;\;\text{and}\;\; L_j=M(A_j) \text{ for } j=1,\dots, m,\]
where $A_j$ is a $n \times n$ matrix with real entries 
such that $(A_j + iI)$ is invertible,
and $M(A_j) = (A_j + iI) \R^n$ for $j=1,...,m$. This formulation was given by Weinstock \cite{weinstock1988polynomial}.

The polynomial convexity of $L_1 \cup L_2$ is completely classified:
Weinstock \cite{weinstock1988polynomial} showed that $\R^n \cup M(A_1)$ is polynomially convex if and only if the matrix $A_1$ does
not have any purely imaginary eigenvalues of modulus greater than one.
Carleman approximation on this union is given by Manne when the Weinstock condition holds:
\begin{result}[Manne \cite{manne1994carlemantwo}]\label{res-Manne}
The union $\R^n\cup M(A)$ allows Carleman approximation if and only if $A$ has no purely imaginary eigenvalue of modulus greater than one.
\end{result}

We show that \Cref{res-Manne} holds for the union of finitely many pairwise transverse totally real subspaces:

\begin{theorem}
  \label{thm:main}
  Let $L_1, ... , L_m \subset \C^n$ be maximal totally real subspaces such that
    $L_k \cap L_j = \{0\}$ for all $k\neq j \in \{1, ..., m\}$ and
  $\bigcup_{j=1}^m L_j$ is locally polynomially convex at the origin.
  Then $\bigcup_{j=1}^m L_j$ allows Carleman approximation by entire functions in $\C^n$.
\end{theorem}

\begin{remark}
    For $j=1,\dots, m$; since $L_j$ is a linear subspace, local polynomial convexity of the union at the origin implies that any compact subset of $\bigcup_{j=1}^m L_j$ is polynomially convex. In turn, this implies that the closed set $\bigcup_{j=1}^m L_j$ is polynomially convex.
\end{remark}
In general, characterizing polynomial convexity of union of finitely many totally real subspaces is difficult. 
However, an open set of three tuples of totally real planes whose union is polynomially convex is given by Gorai \cite{goraithreeplanes}.
For union of finitely many totally real planes there are some classes where polynomial convexity is shown \cite{SG2014}. Using these polynomial convexity results we now provide some corollaries of \Cref{thm:main} which gives some concrete situations where the Carleman approximation holds. For that we need a definition due to Florentino \cite{Floren2009}: {\em A sequence $\mathscr{A}=\{A_1,\dots, A_m\}$ of $2\times 2$ matrices with real entries is said to be reduced if there are no commuting pairs among its terms. A subsequence $\mathscr{B}\subset \mathscr{A}$ is called a reduction of $\mathscr{A}$ if $\mathscr{B}$ is reduced and is obtained from $\mathscr{A}$ by deleting some of its terms. The reduced length of $\mathscr{A}$ is the greatest positive integer $l$ such that there exists a reduction of cardinality $l$.}

\begin{corollary}\label{coro-realevfinite}
    Let $A_j\in\rea^{2\times 2}$, $j=1,\dots, m$ be such that for each $j=1,\dots, N$ $\sigma(A_j)\subset\rea$ and
    $\det[A_j,A_k]=0$ for all $1\leq j<k\leq m$. Let $l$ be the reduced length of $\{A_1,\dots, A_m\}$. If $l=3$ then assume $Tr (ABC-CBA)=0$ for all $A,B,C$ lies in a maximal reduction. Then $\bigcup_{j=1}^mM(A_j)\cup\rea^2$ allows the Carleman approximation.
\end{corollary}

\begin{corollary}\label{coro-nonrealevfinite}
    Let $A_j\in\rea^{2\times 2}$, $j=1,\dots, m$ be such that  $\sigma(A_j)\subset\cplx\setminus\rea$ for each $j$ and
    $\det[A_j,A_k]=0$ for all $1\leq j<k\leq m$. $\bigcup_{j=1}^mM(A_j)\cup\rea^2$ allows the Carleman approximation
    if and only if there is no pair $(j,k), j\neq k, 0\leq j<k\leq m$, satisfying $V_j=cV_k$ for some $c>0$, where $V_j= \left(|\lambda_j|^2-1,2\rl(\lambda_j)\right)$, $\lambda_j\in \sigma(A_j)$ for all $j\geq 1$, $V_0=(1,0)$. 
\end{corollary}
There are several situation when the union of three pairwise transverse totally real planes allows the Carleman approximation. Local polynomial convexity of such unions are described in \cite{goraithreeplanes} and \cite{gorai2019certain}. To mention these as corollaries we need few notations from \cite{goraithreeplanes} and \cite{gorai2019certain}. For a pair of $2\times 2$ matrices $(A_1,A_2)$ we consider the functions: 
\begin{align*}
\Theta(A_1,A_2) &= \det A_1 (Tr A_2)^2+ Tr A_1A_2 (Tr A_1A_2-Tr A_1 Tr A_2).\\
\Lambda (A_1,A_2)&= a\det A_1A_2 - \dfrac{1}{4}(Tr A_1 Tr A_2)^2.\\
\beta (A_1,A_2)&= 4\det A_1A_2 - Tr A_1A_2(Tr A_1A_2-Tr A_1 Tr A_2)
-\dfrac{1}{4}(Tr A_1 Tr A_2)^2.
\end{align*}

\begin{corollary}\label{coro-3planes-real-ev}
Let $A_1,A_2\in\rea^{2\times 2}$ such that $\sigma(A_j)\subset\rea\;\forall j$ and $\sigma((A_1A_2+\id)(A_1-A_2)^{-1})\cap \{it\in\cplx: |t|\geq 1\}=\varnothing$. $M(A_1)\cup M(A_2)\cup\rea^2$ allows the Carleman approximation if one of the following holds:
\begin{itemize}
    \item [(i)] $\det A_j\det[A_1,A_2]>0 \;\forall j=1,2.$
    \item[(ii)] $\det A_j\det[A_1,A_2]<0$ and $\det A_j \Theta(A_j,A_{j^C})<0$ for some $j$ 
\end{itemize}
    \end{corollary}

    \begin{corollary}
        Let $A_1,A_2\in\rea^{2\times 2}$ such that $\sigma(A_1)\subset\rea$, 
        $\sigma(A_2)\subset\cplx\setminus(\rea\cup\{it\in\cplx: |t|\geq 1\})=\varnothing$ and $\sigma((A_1A_2+\id)(A_1-A_2)^{-1})\cap \{it\in\cplx: |t|\geq 1\}=\varnothing$. $M(A_1)\cup M(A_2)\cup\rea^2$ allows the Carleman approximation if one of the following holds: 
        \begin{itemize}
            \item [(i)] $\det A_1\det[A_1,A_2]<0$ and $\det A_1 \Theta(A_1,A_2)<0$.
            \item[(ii)] $\Theta(A_1,A_2)<\Lambda(A_1,A_2)$. 
        \end{itemize}
    \end{corollary}

    \begin{corollary}
       Let $A_1,A_2\in\rea^{2\times 2}$ such that $\sigma(A_j)\subset\cplx\setminus(\rea\cup\{it\in\cplx: |t|\geq 1\})\;\forall j$ and $\sigma((A_1A_2+\id)(A_1-A_2)^{-1})\cap \{it\in\cplx: |t|\geq 1\}=\varnothing$. $M(A_1)\cup M(A_2)\cup\rea^2$ allows the Carleman approximation if 
       $\Theta(A_1,A_2)<\Lambda(A_1,A_2)$.
    \end{corollary}

    \begin{corollary}
        Let $A_1,A_2\in\rea^{2\times 2}$ such that $\sigma(A_j)\subset\cplx\setminus \{it\in\cplx: |t|\geq 1\}\;\forall j$ and $\sigma((A_1A_2+\id)(A_1-A_2)^{-1})\cap \{it\in\cplx: |t|\geq 1\}=\varnothing$. Assume $\det [A_1, A_2]>0$.
        $M(A_1)\cup M(A_2)\cup\rea^2$ allows the Carleman approximation if one of the following holds: 
        \begin{itemize}
            \item [(i)] $\det A_1=0$ and $\det A_2\geq 0$.
            \item[(ii)] $|\det A_j|\leq 1$ and $\det A_j<0$ for $j=1,2$.
            \item[(iii)] $\det A_j<0$ for $j=1,2$, and 
            $\beta(A_1,A_2)>\min \{\det A_2(Tr A_1)^2, \det A_1 (Tr A_2)^2\}.$  
        \end{itemize}
    \end{corollary}
To investigate this deeper for three maximal totally real subspaces:
We will look at the polynomial convexity of three totally real planes in $\C^2$ in
\Cref{thm:three-totally-real-subspaces-part-1} and \Cref{thm:three-totally-real-subspaces-part-2}. Using these we have the following corollaries of \Cref{thm:main}.

\begin{corollary}
    Let $A_1,A_2\in\rea^{2\times 2}$ such that $\sigma((A_1A_2+\id)(A_1-A_2)^{-1})\cap \{it\in\cplx: |t|\geq 1\}=\varnothing$. Assume $\det [A_1, A_2]<0$ and $\det A_j<0$ for $j=1,2$. Then $M(A_1)\cup M(A_2)\cup\rea^2$ allows the Carleman approximation.
\end{corollary}

\begin{corollary}
    Let $A_1,A_2\in\rea^{2\times 2}$ such that $\sigma(A_j)\subset\cplx\setminus \{it\in\cplx: |t|\geq 1\}\;\forall j$ and $\sigma((A_1A_2+\id)(A_1-A_2)^{-1})\cap \{it\in\cplx: |t|\geq 1\}=\varnothing$. Assume $\det [A_1, A_2]>0$, $-1<\det A_1<0$ and $\det A_2>0$. Then $M(A_1)\cup M(A_2)\cup\rea^2$ allows the Carleman approximation.
\end{corollary}

\subsection{The union of two Lipschitz graphs}The other problem we study in this paper is Carleman approximation for union
of two Lipschitz graphs. Our initial interest was to look at
union of two totally real submanifolds but showing polynomial convexity
is challenging compared to showing local polynomial convexity.

Let $\psi \in \smoo^1(\mathbb{R}^n,\mathbb{R}^n)$ follows the Lipschitz condition, i.e., for all $x_1, x_2 \in \mathbb{R}^n$
\[
  \left| \psi(x_1) - \psi(x_2) \right| \leq \alpha | x_1 - x_2 |
\]
where $\alpha \in [0,1)$. Let $\Gamma_\psi = \{ x+i\psi(x) : x\in\mathbb{R}^n \}$ denote
the graph of $\psi$ over the imaginary coordinates.
We say a set $\L$ is a Lipschitz graph if $\L$ can be mapped to $\Gamma_\psi$
under an invertible $\mathbb{C}$-linear map for some $\psi$ defined above.
We can locally identify a totally real submanifold with a Lipschitz graph.
We will discuss some interesting properties of Lipschitz graphs.

\begin{result}\cite[Theorem~1.6.9]{stout2007polynomial}\label{res-manneapprox}
Given $K \subset \Gamma_\psi$ compact, $\mathscr{P}(K) = \smoo(K)$.
\end{result}
From \Cref{res-manneapprox} it follows that any compact $K\subset \L$ is polynomially convex.
Manne shows that a Lipschitz graph allows Carleman approximation in \cite[Section 5]{manne1992some}.
\begin{result}[Manne] \label{thm:carleman-one-lipschitz}
Let $\L$ be a Lipschitz graph as follows
$$
  \L = \{x + i\psi(x): x\in \R^n\},
$$
where $\psi: \R^n \to \R^n$ is $\smoo^1$ with the Lipschitz condition
$
\|\psi(x) - \psi(y)\| \leq \alpha \|x-y\|
$
for all $x,y \in \R$ and $0\leq \alpha < 1$. Then $\L$ allows Carleman approximation.
\end{result}

This article will study Carleman approximation on the union of two
Lipschitz graphs $\L_1$ and $\L_2$
which intersect only at the origin such that $\L_1 \cup \L_2$ is locally polynomially convex.
We also discuss the polynomial convexity of $\L_1 \cup \L_2$ along the way.
Let $\L_1$ and $\L_2$ be two Lipschitz graphs that intersect only at the origin.
After $\C$-linear change of coordinates, we can represent them as
\begin{eqnarray} \label{twoLipschitz}
  \L_1 &=& \{x + i\psi_1(x): x\in \R^n\}, \\
  \L_2 &=& (A+iI)\{x+ i\psi_2(x): x\in \R^n\} \nonumber \\
  &=&  \{ (A+iI)x+i(A+iI)\psi_2(x): x\in \R^n\},
\end{eqnarray}
where $A$ is a real matrix such that $(A+i)$ is invertible
and $\psi_1, \psi_2 \in \smoo^1(\mathbb{R}^n,\mathbb{R}^n)$
which follow the Lipschitz condition, i.e., for all $x_1, x_2 \in \mathbb{R}^n$
\[
  \left| \psi_j(x_1) - \psi_j(x_2) \right| \leq \alpha_j | x_1 - x_2 |
\]
where $\alpha_j \in [0,1)$ for $j = 1,2$ and $\psi_j(0) = 0$.

Further, if $D\psi_1(0) = 0$ and $D\psi_2(0) = 0$, then the above represents two Lipschitz graphs
whose tangent spaces at the origin are $\mathbb{R}^n$ and $(A+iI)\mathbb{R}^n$ respectively. We ask the question: Under the assumption of polynomial convexity of the union of the tangent spaces, {\em whether $\L_1 \cup \L_2$ polynomially convex?} Answer to this seems to depend on the Lipschitz constant $\alpha$. Hence, we also ask: {\em Can we choose appropriate $\alpha_j$ to ensure this?} Since the computations in higher dimensions are a bit complicated we focus here for $n=2$ only.

\begin{theorem} \label{prop:poly-cnv-graphs}
  Let $\L_1$ and $\L_2$ be Lipschitz graphs on $\C^2$ 
  which intersect only at the origin and are given by 
\begin{eqnarray*}
    \L_1 &=& \{x + i\psi_1(x): x\in \R^2\}, \\
    \L_2 &=& \{(A+iI)x+i(A+iI)\psi_2(x): x\in \R^2\}.
\end{eqnarray*}
	Suppose $A$ is a real $2\times 2$ matrix that doesn't have an eigenvalue which is purely imaginary and has modulus greater than one (Weinstock criteria).
	Then there exist constants $\alpha_1, \alpha_2 > 0$ determined by $A$ such that
	if $\psi_j$ is Lipschitz with coefficient $\alpha_j$ for $j=1,2$, then
  Then $\L_1 \cup \L_2$ is polynomially convex.
\end{theorem}
Provided a $2\times 2$ matrix $A$ that satisfies the Weinstock condition, one can determine Lipschitz constants $\alpha_1,\alpha_2$
by following the details in the proof of the above theorem (\Cref{subsection:poly-cnx-union-of-lipschitz-graphs}).

%
%
%
\smallskip

Suppose we take $\L_1$ and $\L_2$ as in \Cref{prop:poly-cnv-graphs}. Then by
the Lipschitz condition, setting one point to $0$, we get $\|\psi_j(x)\| \leq \alpha_j \|x\|$, for $x\in \R^n$. So we have
\begin{eqnarray*}
  \L_1 &=& \{x+ i\psi_1(x): x\in \R^n\} \\
  &\subset& \{x+iy \in \C^n: \|y\|\leq \alpha_1 \|x\| \}
\end{eqnarray*}
Similarly,
\begin{eqnarray*}
  \L_2 &=& \{(A+iI)(x+ i\psi_2(x)): x\in \R^n\} \\
  &\subset& (A+iI)\{x+iy \in \C^n: \|y\|\leq \alpha_2 \|x\| \}
\end{eqnarray*}

As $(A+iI)\R^n \cap \R^n = \{0\}$, we can choose $\alpha_1>0$ and $\alpha_2>0$ small enough such that
\begin{equation}
  C_1 \cap C_2 = \{0\} \label{eqn:plane-angular},
\end{equation}
where $C_1 =    \left\{x+iy \in \C^n: \|y\|\leq \alpha_1 \|x\| \right\} $
and $C_2 =   (A+iI)\left\{x+iy \in \C^n: \|y\|\leq \alpha_2 \|x\| \right\}$

\begin{theorem}
  \label{thm:lipschitz-main}
  Let $\L_1$ and $\L_2$ be Lipschitz graphs as in \Cref{prop:poly-cnv-graphs}.
  Then $\L_1 \cup \L_2$ allows Carleman approximation by entire functions in $\C^2$
  when $\alpha_1>0$ and $\alpha_2>0$ are chosen such that
  \begin{itemize}
    \item $\alpha_1 < 1/3$ and $\alpha_2 < 1/3$
    \item $\alpha_1$ and $\alpha_2$ small enough such that (\ref{eqn:plane-angular}) holds.
    \item $\alpha_1$ and $\alpha_2$ small enough such that \Cref{prop:poly-cnv-graphs} holds
      (for the polynomial convexity of $\L_1 \cup \L_2$)
  \end{itemize}
\end{theorem}

\begin{remark}
The above theorem can be generalized to
Lipschitz graphs in $\C^n$ for $n\geq 3$.
Then the proof of \Cref{thm:lipschitz-main} follows
if we obtain polynomial convexity of the union of two Lipschitz graph in $\C^n$ as in \Cref{prop:poly-cnv-graphs}. 
Local polynomial convexity of such a union of Lipschitz graphs in $\C^n$ has been shown by Shafikov and Sukhov \cite[Theorem 4.2]{shafikov-sukhov2016}.
But showing (global) polynomial convexity for
higher dimension is quite challenging.
\end{remark}

\begin{remark}
Given the polynomial convexity of
the union of finitely many Lipschitz graph, one
can discuss Carleman approximation on this union.
This can be shown by combining the ideas in the proof of
\Cref{thm:main} and \Cref{thm:lipschitz-main}.
\end{remark}

\subsection{Surface in $\C^2$ with a hyperbolic complex point}

Let $M \subset \C^2$ be a smooth real surface which is totally real except at $0\in M$. Suppose $0\in M$ is a non-degenerate complex point.
Such a complex point occurs in a generic embedding of $M \hookrightarrow \C^2$.
Bishop \cite{bishop1965} has given a normal form for $M$ near the complex point $0$.
After biholomorphic change of coordinates, we can represent $M$ near $0$ by
$$
    z_2 = z_1 \bar z_1 + \gamma (z_1^2 + \bar z_1^2) + F(z_1),
$$
where $\gamma > 0$, $z_1 \in \overline{D_R(0)}\subset C$,
and $F: \overline{D_R(0)} \to \C$ is a smooth function that vanishes to order three at $0\in M$.
We call the complex point $0\in M$ elliptic if $\gamma < \frac{1}{2}$, parabolic if $\gamma = \frac{1}{2}$
and hyperbolic if $\gamma > \frac{1}{2}$.

Bishop \cite{bishop1965} has shown that a family of analytic discs can be attached to $M$ near an elliptic point.
Hence, we get an obstruction to polynomial convexity at elliptic points.
Forstneri\v c and Stout \cite[Theorem VI]{forstneric-stout1991} have shown that
$M \subset \C^2$ is locally polynomially convex at a hyperbolic complex point.
They look at the preimage of $M$ through a proper holomorphic map from $\C^2$ to $\C^2$.
When $F$ is bounded in $\mathcal C^2$-norm,
the preimage of this surface can be realized as a union of Lipschitz graphs \cite[Theorem 1.3]{alagandala2025}.
We will show that $M$ allows Carleman approximation under a certain bound on the $\mathcal C^2$-norm of $F$.

\begin{theorem}
	\label{thm:hyperbolic-carleman-main}
	Let $M$ be a surface in $\C^2$ given by the following
	$$
			z_2 = z_1 \bar z_1 + \gamma (z_1^2 + \bar z_1^2) + F(z_1),
	$$
	where $\gamma > \frac{1}{2}$, $z_1 \in \C$,
	and a smooth function $F: \C \to \C$ that vanishes to order three at $0\in M$.
	Suppose $\|F\|_{\mathcal C^2(\C)}$  is small enough so that:
	\begin{equation}
		\|F\|_{\mathcal C^2(\C)} \leq \min\left\{
			\frac{(2\gamma -1)^3}{2^{14}\gamma^3},
			\left( \frac{2\gamma}{\sqrt{(2\gamma)^2 - 1}} - 1\right)
				\frac{{(2\gamma -1)^4} }{2^{10}\gamma (2\gamma +1)},
			\frac{1}{3} \frac{{(2\gamma -1)^3} }{2^{10}\gamma (2\gamma +1)}
		\right\}.
		\label{eqn:F-c2-norm-min-ineq}
	\end{equation}
	Then $M$ allows Carleman approximation.
\end{theorem}

\begin{example}
	The surface $M \subset \C^2$ given by
	$$
			z_2 = z_1 \bar z_1 + \gamma (z_1^2 + \bar z_1^2) +  \Omega(\gamma) \cos(\text{Re} (z_1)),
	$$
	allows Carleman approximation for $\gamma > \frac{1}{2}$.
	Here $\Omega(\gamma)$ is the minimum in \Cref{eqn:F-c2-norm-min-ineq}.
\end{example}

Some words about the layout of this paper: In Section~\ref{sec:preli} we recall some machinery to determine polynomial convexity.
In \Cref{section:three-totally-real-planes}, we look at polynomial convexity of
the union of three totally real subspaces that are mutually transverse in $\C^2$.
Then in \Cref{section:union-of-finitely-many-totally-real-subspaces},
we focus on Carleman approximation on the union of finitely many totally real
subspaces. 
In \Cref{section:union-of-lipschitz-graphs}, we study polynomial convexity of the
union of two Lipschitz graphs in $\C^2$ and proceed to show Carleman
approximation on this union.
In \Cref{section:surfaces-hyperbolic-point}, we analyze how Carleman
approximation behaves under proper holomorphic maps and
show a family of surfaces with a hyperbolic complex point that allow Carleman approximation.
In \Cref{section:ck-smoothness-remarks}, we remark about $\mathcal C^k$-smooth
Carleman approximation.

\section{Technical preliminaries}\label{sec:preli}
A compact set $K\subset \L_1 \cup \L_2$ can be seen as $K_1 \cup K_2$ where $K_j \subset \L_j$ is compact.
As $K_1$ and $K_2$ are polynomially convex, Kallin's lemma could be used to check if
$K = K_1 \cup K_2$ is polynomially convex.
\begin{lemma}[{Kallin's Lemma  \cite[Theorem 1.6.19]{stout2007polynomial}}] \label{lemma:Kallins-poly}
  Let $K_1$ and $K_2$ be two polynomially convex compact set in $\C^n$.
  Let $P$ be a holomorphic polynomial on $\C^n$ such that
  $Y_j = \widehat{(P(K_j))}$ for $j=1,2$ meet at at most at the origin,
  which is a boundary point for both the sets.
  If the set $P^{-1}(0) \cap (K_1 \cup K_2)$ is polynomially convex,
  then $K_1 \cup K_2$ is polynomially convex.
\end{lemma}

Here is an approximation result for compactly supported continuous functions on Lipschitz graphs of $\R^n$.
The details are present in \Cref{thm:carleman-one-lipschitz} of Manne \cite{manne1992some}. 

\begin{proposition}
  \label{prop:lipshitz-one-approx}
  Let $\L = \{(x+i\psi(x)): x\in \R^n\}$ where $\psi \in \smoo^1(\R^n, \R^n)$ and satisfies the Lipschitz condition with $\alpha<1$.
  Let $f \in \smoo(\L)$ have compact support and $\epsilon > 0$.
  Define
  $$
    h_t(z) :=  \left(\frac{1}{t\sqrt{\pi}}\right)^n \int_{\L} f(u) \exp\left({-\frac{(z-u)^2}{t^2}}\right)du
  $$ 
  Then $h_t \in \holo(\C^n)$ and
  $h_t \to f$ uniformly on $\L$ as $t\to 0^+$.

  Further, for $\delta>0$, define
  $$
    A_\delta := \{z\in\C^n: Re(z-w)^2 \geq \delta \quad \forall w \in \text{supp}(f) \}
  $$
  Then $h_t \to 0$ uniformly on $A_\delta$.
\end{proposition}
The following lemma will be used in \Cref{thm:three-totally-real-subspaces-part-1}.
\begin{lemma}
  \label{lemma:am-gm}
  Let $a,b, c \in \R$ such that
  $ab > c^2$.
  Then $ax^2 + by^2 + 2cxy = 0$ if, and only if, $(x,y) = (0,0)$.
  Where $x,y\in\R$.
\end{lemma}

\begin{proof}
  The proof is an application of the AM-GM inequality.
  Assume $a>0$. Then $b>0$ and $\sqrt{ab} >  |c|$.
  \begin{itemize}
  \item
    Case 1: $x \neq 0, y\neq 0$:\\
    By AM-GM inequality on $ax^2$ and $by^2$, we get:
    $$
    ax^2 + by^2 \geq 2 \sqrt{ab}|xy| > 2 |c| |xy|.
    $$
  \item
    Case 2: either $x = 0$ or $y=0$ but not both:\\
      Say $x=0$ and $y>0$, then
      $$
      ax^2 + by^2 = by^2 > 0 = 2 |c| |xy|.
      $$
  \end{itemize}

  If we have $ax^2 + by^2 + 2c xy = 0$ then $ax^2 + by^2 = 2|c| |xy|$
  which only happens when $(x,y) = (0,0)$.

  When $a < 0$, we can replace $a, b, c$ with $-a, -b,-c$ respectively.
  Then the same proof follows.
\end{proof}

\section{Polynomial convexity of three totally real planes}
\label{section:three-totally-real-planes}

In this section,
we will analyze two condition for the polynomial convexity
of three maximal totally real planes given that union of any two 
of them is polynomially convex.

\begin{theorem}
    \label{thm:three-totally-real-subspaces-part-1}
  Let $P_0$, $P_1$, $P_2$ be maximal totally real subspaces in $\C^2$ as:
  \begin{eqnarray*}
    P_0 &=& \R^2, \\ 
    P_1 &=& (A_1 + iI)\R^2, \\
    P_2 &=& (A_2 + iI)\R^2,
  \end{eqnarray*}
  where $A_1, A_2 \in \R^{2\times 2}$,
  the union of any two of $\{P_0, P_1, P_2\}$ is locally polynomially convex at the origin
  and the following
  holds:
  \begin{itemize}
    \item $\det{[A_1,A_2]} < 0$,
    \item $\det{A_j} < 0$, for $j=1,2$.
  \end{itemize}
  Then $P_0 \cup P_1 \cup P_2$ is locally polynomially convex at the origin.
\end{theorem}

\begin{proof}
  As $\det A_1 < 0$, it has distict non-zero real eigenvalues of different signs.
  By \cite[Lemma 2.4]{goraithreeplanes}, 
  after suitable change of coordinates, we get
  $$
  A_1 = \begin{bmatrix}
    \lambda_1 & 0 \\
    0 & \lambda_2 
  \end{bmatrix},
  $$
  and
  $$
    A_2 = \begin{bmatrix}
      s_1 & q \\
      q & s_2 
    \end{bmatrix}
    \text{ or }
    \begin{bmatrix}
      s_1 & - q \\
      q & s_2 
    \end{bmatrix},
  $$
  where $\lambda_1, \lambda_2 \in \R$.
  With $\lambda_1 > 0$ and $\lambda_2 < 0$.
  Where $s_1, s_2, q \in \R$.
  But since $\det [A_1, A_2] < 0$, it is forced that
  $$
  A_2 = \begin{bmatrix}
      s_1 & - q \\
      q & s_2 
    \end{bmatrix}.
  $$

  We can write
  \begin{eqnarray*}
    P_0 &=& \{ (x,y) : x,y \in \R \}, \\
    P_1 &=& \{ ((\lambda_1 +i) x,(\lambda_2+i) y) : x,y \in \R \}, \\
    P_2 &=& \{ ((s_1+i) x - qy ,qx + (s_2+i) y) : x,y \in \R \}.
  \end{eqnarray*}
  
  Let $
  K_j = \overline{B(0,1)} \cap P_j 
  $ for $j=0,1,2$. We need polynomially convexity of $K_0 \cup K_1 \cup K_2$.
  For this we will apply Kallin's \Cref{lemma:Kallins-poly} to $K_0$ and $K_1 \cup K_2$.
  
  Consider the polynomial $$p (z) = z_1^2 - z_2^2.$$
  Suppose $z \in K_0$, then $\Im (p(z)) = 0$. We have $p(K_0) \subset \{x: x\in \R \}$.
  And $$p^{-1}\{0\} \cap (K_0) = \{(x,\pm x): x\in \R\} \cap \overline{B(0,1)}.$$
  Suppose $z \in K_1$, for some $z = ((\lambda_1 +i) x,(\lambda_2+i) y)$ for some $x,y \in \R$.
  Then $\Im(p(z)) = 2(\lambda_1 x^2 + (- \lambda_2) y^2)$.
  Note that $\lambda_1 > 0$ and $\lambda_2 < 0$.
  We have $\Im(p(z)) > 0$ when $(x,y) \neq (0,0)$.
  And $p^{-1}\{0\} \cap K_1 = \{(0,0)\}$.

  This shows $\widehat{p(K_0)}\cap \widehat{p(K_1)} = \{0\}$ and
  $$p^{-1}\{0\} \cap (K_0 \cup K_1) = \{(x,\pm x): x\in \R\} \cap \overline{B(0,1)}.$$

  Suppose $z \in K_2$, say $z= ((s_1+i) x - qy ,qx + (s_2+i) y)$. Then 
  $$\Im(p(z)) = -4qxy + 2s_1 x^2 - 2s_2 y^2.$$
  We have $\det A_2 = s_1 s_2 + q^2 < 0$.
  Set $a = s_1$, $b= - s_2$ and $c = -q$. We have 
  $\Im(p(z)) = 2(cxy + a x^2 +b y^2)$ and $ab > c^2$. By \Cref{lemma:am-gm} we
  get $\Im(p(z))=0$ if, and only if, $(x,y)=(0,0)$.
  Hence $\Im(p(z)) \neq 0$ unless $z=(0,0)$.
  And $p^{-1}\{0\} \cap K_2 = \{(0,0)\}$.

  This shows $\widehat{p(K_0)}\cap \widehat{p(K_2)} = \{0\}$ and 
  $$ p^{-1}\{0\} \cap (K_0 \cup K_2) = \{(x,\pm x): x\in \R\} \cap \overline{B(0,1)}.$$

  Let $K = K_1 \cup K_2$.
  From the above two results $\widehat{p(K_0)}\cap \widehat{p(K)} = \{0\}$.
  And $p^{-1}\{0\} \cap (K_0 \cup K) = \{(x,\pm x): x\in \R\} \cap \overline{B(0,1)}$.
  Which is polynomially convex as any compact subset of $\R^2$ is polynomially convex in $\C^2$.


\end{proof}

\begin{theorem} \label{thm:three-totally-real-subspaces-part-2}
  Let $P_0$, $P_1$, $P_2$ be maximal totally real subspaces in $\C^2$ as:
  \begin{eqnarray*}
    P_0 &=& \R^2, \\ 
    P_1 &=& (A_1 + iI)\R^2, \\
    P_2 &=& (A_2 + iI)\R^2,
  \end{eqnarray*}
  where $A_1, A_2 \in \R^{2\times 2}$,
  the union of any two of $\{P_0, P_1, P_2\}$ is locally polynomially convex at the origin
  and the following
  holds:
  \begin{itemize}
    \item $\det{[A_1,A_2]} > 0$,
    \item $-1< \det{A_1} < 0$,
    \item $\det{A_2} > 0$.
  \end{itemize}
  Then $P_0 \cup P_1 \cup P_2$ is locally polynomially convex at the origin.
\end{theorem}

\begin{proof}
  Since $\det{A_1} < 0$, the matrix $A_1$ has real, non-zero
  eigenvalues of opposite signs. And $\det{[A_1,A_2]} > 0$.
  By \cite[Lemma 2.5]{goraithreeplanes},
  after a suitable change of coordinates:

  $$
    A_1 = \begin{bmatrix}
      \lambda_1 & 0 \\
      0 & \lambda_2
    \end{bmatrix}   \quad
    A_2 = \begin{bmatrix}
      s_1 & t \\
      t & s_2
    \end{bmatrix}
  $$

  We can write
  \begin{eqnarray*}
    P_0 &=& \{ (x,y) : x,y \in \R \}, \\
    P_1 &=& \{ ((\lambda_1 +i) x,(\lambda_2+i) y) : x,y \in \R \}, \\
    P_2 &=& \{ ((s_1+i) x + ty ,tx + (s_2+i) y) : x,y \in \R \}.
  \end{eqnarray*}
  Define $$p(z) := z_1^2 +  z_2^2.$$
  Let $
  K_j = \overline{B(0,1)} \cap P_j 
  $ for $j=0,1,2$. We wish to show $K_0 \cup K_1 \cup K_2$ is polynomially convex.

  Say $z \in K_0$, then $\Im (p(z)) = 0$. We have $p(K_0) \subset \{x: x\in \R \}$.

  Say $w\in K_1$ with the form
  $w = \left(  
  (\lambda_1 + i)x, (\lambda_2 + i)y
  \right)$. Then
  calculating  $p(w)$:
  \begin{eqnarray*}
    \Re{p(w)}&=& (\lambda_1^2-1) x^2 + (\lambda_2^2 -1)y^2, \\
  \Im{p(w)}&=& 2\lambda_1 x^2 + 2\lambda_2 y^2.
  \end{eqnarray*}

  If $\Im p(w) = 0$, then $x^2 = -\frac{\lambda_2}{\lambda_1} y^2$.
  In that case:
  \begin{eqnarray*}
    \Re p(w) &=& (\lambda_1^2 - 1)x^2 + (\lambda_2^2 - 1) y^2\\
    &=& (\lambda_1^2 - 1)\left(-\frac{\lambda_2}{\lambda_1} y^2 \right) + (\lambda_2^2 - 1) y^2\\
    &=& \frac{1}{\lambda_1} (-\lambda_1^2\lambda_2 + \lambda_2 + \lambda_1\lambda_2^2 - \lambda_1) y^2\\
    &=& \frac{1}{\lambda_1} (\lambda_1\lambda_2 (\lambda_2 - \lambda_1) + \lambda_2 - \lambda_1) y^2\\
    &=& \frac{1}{\lambda_1}(\lambda_2 - \lambda_1) (\lambda_1\lambda_2  + 1) y^2.
  \end{eqnarray*}
  Since $\det A_1 > -1$, $\lambda_1 \lambda_2 + 1 > 0$. Hence $\Re p(w) < 0$ when $\Im p(w) =0$ unless $w = (0,0)$.
  $$
  \widehat{p(K_2)} \cap \widehat{p(K_0)} = \{0\} \quad \& \quad p^{-1}\{0\} \cap (K_0 \cup K_2) = \{(0,0)\}.
  $$

  Say $w\in K_2$ with the form
  $w = \left((s_1+i)x+ty, tx+(s_2+i)y\right)$. Then
  calculating $\Im{p(w)}=2s_1x^2 + 2s_2y^2+4txy$.
  Since $\det{A_2} > 0$, we have $s_1 s_2 > t^2$.
  Set $a=s_1$, $b=s_2$ and  $c=t$. Then $ab>c^2$.
  By \Cref{lemma:am-gm} we get
  $\Im{p(z)} = 0$ if, and only if, $z=0$.
  $$
  \widehat{p(K_1)} \cap \widehat{p(K_0)} = \{0\} \quad \& \quad p^{-1}\{0\} \cap (K_0 \cup K_1) = \{(0,0)\}.
  $$
  Finally, take $K=  K_1 \cup K_2$ , then
  $$
  \widehat{p(K)} \cap \widehat{p(K_0)} = \{0\} \quad \& \quad p^{-1}\{0\} \cap (K_0 \cup K) = \{(0,0)\}.
  $$
  As $K$ is polynomially convex. By Kallin's \Cref{lemma:Kallins-poly}, $K_0 \cup K_1 \cup K_2$ is polynomially convex.

\end{proof}

\section{The union of finitely many totally real subspaces}
\label{section:union-of-finitely-many-totally-real-subspaces}

The following proof is inspired by techniques in Manne \cite{manne1994carlemantwo}.

\begin{proposition}
  \label{prop:sqeeze}
  Let $L_1, ..., L_m$ be as in \Cref{thm:main}. Fix $R > 1$, and let
  $$
  K =  \{z\in \cup_{j=1}^m L_j : |z| \leq R\} \text{ and }
  K_j =\{z\in  L_j : 1 \leq |z| \leq R\}.
  $$
  Then there exists $\delta > 0$ such that
  for any continuous function $f \in \smoo(\cup_{j=1}^m L_j) $ with compact support in $\cup_{j=1}^m K_j$ and 
  any $\epsilon > 0$ there is an entire function $h \in \holo(\C^n)$ such that
  $\|h-f\|_{K} < \epsilon $ and $|h(z)| < \epsilon$ for $|z|\leq \delta$.
\end{proposition}

\begin{proof}
  
  Fix $j\in \{1,...,m\}$. Let $\phi_j$ represent a $\C$-linear automorphism of $\C^n$ that maps the maximal totally real plane $L_j$ to $\R^n$.
  Then $\phi_j(K_j)$ is a subset of $\R^n$.
  We may use \Cref{prop:lipshitz-one-approx} with $\psi \equiv 0$ for approximating $f$ on $\phi_j(K_j)$.
  Before doing so, we will define some neighborhoods of $K_j$.
  Choose $\gamma > 0$ sufficiently small such that
  $ d(\phi_j(K_j), \phi_j(L_k)) > 3 \gamma $ holds
  for  all $k \in \{1,...,m\}$ with $k\neq j$.
  This follows as $\phi_j$ is a homeomorphism, 
  $K_j$ is compact, $L_k$ is closed, and $K_j \cap L_k = \varnothing$.

  Define a neighborhood 
  $ \Omega_o = \{ z\in \C^n : d(z,K)<\eta \} $ of $K$
  where $\eta > 0$ is taken small enough to ensure
  $ \phi_j(\Omega_o) \subset \{z\in \C^n : d(z, \phi_j(K)) < \gamma \} $
  for $j \in \{1,...,m\}$. This follows from the continuity of each $\phi_j$.
  
  \newcommand{\U}{\tilde{U}_j}
  \newcommand{\V}{\tilde{V}_j}
  \newcommand{\W}{\tilde{W}_j}
  For $j \in \{1,...,m\}$, define two open sets $\U$ and $\V$  such that
  they cover $\phi_j(K)$:
  $$
  \U = \left\{ z \in \C^n: d\left(z,\phi_j(K)\right) < \gamma \text{ and } d\left(z, \phi_j(K_j)\right) < 2\gamma \right\},
  $$
  $$
  \V = \left\{ z \in \C^n: d\left(z,\phi_j(K)\right) < \gamma \text{ and } d\left(z, \phi_j(K_j)\right) > \frac{3}{2}\gamma \right\}.
  $$

  Observe that, $\phi_j(K_j) \subset \U$, $\phi_j(L_k\cap K) \subset \V$ for $k\neq j$ and
  $0\in \V$.

  \newcommand{\h}{\tilde{h}_j^{(t)}}
  \renewcommand{\a}{\tilde{\alpha}_j^{(t)}}
  \renewcommand{\b}{\tilde{\beta}_j^{(t)}}
  Define $\h: \C^n \to \C^n$ as
  $$
  \h(z)  :=  \left(\frac{1}{t\sqrt{\pi}}\right)^n \int_{\R^n} f\circ \phi_j^{-1}(u) e^{-\frac{(z-u)^2}{t^2}}du.
  $$
  Consider all limits as $t\to 0^+$.
  Now by \Cref{prop:lipshitz-one-approx}, we have $\h \in \holo(C^n)$,
  and $\h \to f\circ \phi_j^{-1}$
  uniformly  on $\phi_j(L_j) = \R^n$.
  Let $\W = \V \cap \U$,
  we will show $\W \subset A_{\frac{\gamma^2}{4}}$ with
  the notation of \Cref{prop:lipshitz-one-approx}; this gives us
  $\h \to 0$ uniformly on $\W$.
  Let $z\in \W$, $u\in \phi_j(K_j) \subset \R^n$,
  say $z=(x_1+i y_1, ..., x_n+i y_n)$ and $u = (u_1, ..., u_n) \in \R^n$.
  \begin{eqnarray*}
    Re(z-u)^2 &=& Re \sum_{l=1}^n (z_l - u_l)^2 \\
    &=& \sum_{l=1}^n (x_l - u_l)^2 - \sum_{l=1}^n y_l^2
  \end{eqnarray*}
  As $z\in\W \subset \V$, $d(z,u) > \frac{3}{2}\gamma$ which gives
  $$
    \sum_{l=1}^n (x_l - u_l)^2 + \sum_{l=1}^n y_l^2 > \frac{9}{4}\gamma^2
  $$
  By definition of $\U$, $d(z,\phi_j(K))<\gamma$ and $d(z,\phi_j(K_j)) < 2 \gamma$.
  By the choice of $\gamma$, $d(\phi_j(K_j),\phi_j(L_k)) > 3\gamma$ for $k\neq j$; this implies
  $d(z,\phi_j(L_k))>\gamma$. This gives us
  $$
    \sum_{l=1}^n y_l^2 =  (d(z,\R^n))^2 = (d(z,\phi_j(L_j)))^2 < \gamma^2.
  $$
  Combining both the inequalities, $\sum_{l=1}^n (x_l - u_l)^2 > \frac{5}{4}\gamma^2$.
  Finally, $$
    Re(z-u)^2 = \sum_{l=1}^n (x_l - u_l)^2 - \sum_{l=1}^n y_l^2
    > \frac{5}{4} \gamma ^2 - \gamma^2 = \frac{\gamma^2}{4}
  $$
  This proves $\W \subset A_{\gamma^2/4}$ and $\h \to 0$ uniformly (\Cref{prop:lipshitz-one-approx}).

  The set $K$ is polynomially convex as it is a compact subset of $\cup L_j$.
  We can find a polynomial polyhedron $\Omega$, which is a pseudoconvex domain, such that $K \subset \Omega \subset \Omega_o$.

  For $j\in\{1,...,m\}$,
  set $U_j = \Omega \cap \U$, 
  $V_j = \Omega \cap \V$ and
  $W_j = V_j \cap U_j = \Omega \cap \phi_j^{-1}(\W)$.
  Then $U_j \cup V_j = \Omega$ and $\h \circ \phi_j \to 0$ uniformly on $W_j$.

  Consider the linear operator $T: \holo(U_j) \bigoplus \holo(V_j) \to \holo(W_j)$
  given by
  $$T(\alpha, \beta) = (\alpha - \beta)|_{W_j}$$
  The Cousin I problem can be solved on the pseudoconvex set $\Omega$ with $\{U_j, V_j\}$;
  This shows the surjectivity of $T$.
  Which allows us to use the open mapping theorem for Fréchet spaces.
  As $\h \circ \phi_j \to 0$ in $\holo(W_j)$, we can solve for $(\a, \b) \in \holo(U_j) \bigoplus \holo(V_j)$
  with
  $$T(\a,\b) = (\a-\b)|_{W_j} = \h \circ \phi_j,$$
  such that $(\a, \b) \to 0$, i.e., $\a \to 0$ in $\holo(U_j)$ and $\b \to 0$ in $\holo(V_j)$.

  \newcommand{\hjt}{h_j^{(t)}}
  \vspace{1em}
  We have seen that $0 \in V_j$ for all $j$.
  Take $\delta > 0$ such that the closed ball $B(0,\delta) \Subset V_j$ for all $j$.
  Define $\hjt \in \holo(\Omega)$ with
  $\hjt = \h \circ \phi_j -\a$ on $U_j$ and $\hjt = -\b$ on $V_j$. This is well defined
  due to the Cousin's I problem.

  Then $\hjt \to 0$ in $\holo(V_j)$; this gives: 
  $\hjt \to 0$ uniformly on $L_k \cap K  \Subset V_j$ for $k\neq j$,
  and $\hjt \to 0$ uniformly on $B(0,\delta)   \Subset V_j$ for $k\neq j$.
  Further,
  $\hjt \to f$ uniformly on compact subsets of
  $L_j \cap V_j$ as $f=0$ on $L_j - K_j$.

  The entire function $\h \circ \phi \to f$ uniformly on $L_j$ and
  $\a \to 0$ in $\holo(U_j)$ which implies
  $\hjt \to f$ uniformly on compact subsets of $L_j \cap U_j$.
  Combining with the previous argument, $\hjt \to f$ uniformly on $L_j \cap K$,
  $\hjt \to 0$ uniformly on $L_k \cap K$ for $k\neq j$,
  and $\hjt \to 0$ uniformly on $B(0,\delta)$.

  \renewcommand{\ht}{h^{(t)}}
  Define $\ht := \sum_{j=1}^{n} \hjt$.
  Then $\ht \to f$ uniformly on $L_j \cap K$ for all $j$.
  Hence, $\ht \to f$ uniformly on $K$ and $\ht \to 0$ uniformly on $B(0,\delta)$.

\end{proof}

\begin{corollary}
  \label{cor:vanish}
  For $R_o>1$, we can choose $\delta$ for which the following holds.
  If $f_o \in \smoo(\cup L_j)$ such that $f_o(z)=0$ for $|z| \leq r$ for some $r>0$.
  Then for any $\epsilon > 0$,
  there exists $h_o \in \holo(\C^n)$ such that $|h_o(z)-f_o(z)|<\epsilon$
  for $z \in \left(\cup L_j \right) \cap  \{z\in \C^n: |z| \leq rR_o\} $, and
  $|h_o(z)|<\epsilon$ for $|z|\leq r\delta$.
\end{corollary}

\begin{proof}
  We define
  $$f(z) := f_o(z/r) \rho_{R_o}(|z|),$$
  where $\rho_{R_o}: [0,+\infty) \to [0,1]$ is a continuous function that is
  $1$ on $[0,R_o]$ and $0$ on $[R_o+1, + \infty)$. Let $R=R_o+1$;
  it follows that $\text{supp}(f_o) \subset \cup K_j$ with the notations of
  \Cref{prop:sqeeze}. By the proposition, we have $\delta>0$ corresponding
  to $R$. Then there is an entire function $h\in \holo(\C^n)$ such that
  $|f - h|_K < \epsilon $ and $|h|_{B(0,\delta)} < \epsilon$.
  Define another entire function $h_o(z) := h(rz)$. We have $|f(z) - h(z)| < \epsilon$ for
  $z \in (\cup L_j) \cap \{z\in\C^n: |z| \leq R\}$.
  Replacing $f$ with $f_o$, we get
  $|f_o(z/r) - h_o(z/r)| < \epsilon$ for
  $z \in (\cup L_j) \cap \{z\in\C^n: |z| \leq R_o\}$.
  Hence,
  $|f_o(z) - h_o(z)| < \epsilon$ for
  $z \in (\cup L_j) \cap \{z\in\C^n: |z| \leq rR_o\}$ and 
  $|h_o(z)| < \epsilon$ for
  $|z| \leq r\delta$.
\end{proof}

\begin{proof}[Proof of \Cref{thm:main}]
  We are given $f\in \smoo(\cup L_j)$ and $\epsilon \in \smoo(\cup L_j, \R_{>0})$.
  Pick $R>1$ and choose $\delta$ from \Cref{cor:vanish}.
  Take an increasing sequence of positive real numbers $\{r_\eta\}$ such that
  $r_{\eta+2} = R r_\eta$. We define continuous functions $\rho_\eta: [0,+\infty) \to [0,1]$
  such that $\rho_\eta(t) = 0$ for $t\in[0,r_\eta]$ and $\rho_\eta(t) = 1$ for $t\in [r_{\eta + 1}, +\infty)$.

  Consider the exhaustion of $\cup L_j$ by compacts $K_\eta = \{z\in \cup L_j : |z| \leq r_\eta R\}$.
  Then the inner closed ball for $K_\eta$ in \Cref{cor:vanish} will be
  $B_\eta = B(0,r_\eta \delta)$. Set $\epsilon_\eta = \min_{t\in K_\eta}{\epsilon(t)}$.

  The set $K_0$ is polynomially convex by local polynomial convexity of $\cup L_j$ at the origin.
  Further, we have $\mathcal P (K_0) = \mathcal C(K_0)$.
  There exists an entire function $h_0$ such that
  $$
    |h_0(z) - f(z)| < \epsilon_0/2 \quad \forall z \in K_0.
  $$
  We want an entire function $h_1$ such that
  $$
    |h_1(z) - \rho_0(|z|)(f(z) - h_0(z) )| < \epsilon_1/2^2 \quad \forall z \in K_1
  $$ and $|h_1|_{B_1} \leq \epsilon_1/2^2$. Such a $h_1$ can be found using
  \Cref{cor:vanish} as  $f_o(z) = \rho_0(|z|)(f(z) - h_0(z) )$ satisfies the conditions
  with $r = r_0$.
  This process can be done inductively to obtain entire function $h_\eta$  such that
  $$
    \left|h_\eta(z) - \rho_\eta(|z|)(f(z) -\sum_{l=1}^{\eta-1} h_l(z) )\right| < \epsilon_\eta/2^{\eta+1} \quad \forall z \in K_\eta
  $$
  and $|h_\eta|_{B_\eta} < \epsilon_\eta/2^{\eta+1}$.
  Define $h = \sum_{\eta \in \mathbb{N}} h_\eta$. Given any compact set $A\subset \C^n$, we can find $\eta$
  such that $A\subset B_\eta$. Then $h$ converges uniformly on $A$ as
  $|\sum_{j>\eta} h_j|_{A} < \epsilon_\eta/2^{\eta}$.
  Hence, $h$ is an entire function.
  We are done if we show $|h(z)-f(z)|< \epsilon(z)$ for $z\in \cup \L_j$.
  Take any $z\in \cup L_j$,
  say $r_{\eta_o} < |z| \leq r_{\eta_o +1}$.

  \begin{equation*}
    |h(z) - f(z)|
    = \left|\sum_{\eta =0}^{\infty} h_\eta(z)  - f(z) \right| \leq  \sum_{\eta > \eta_o}^{\infty} |h_\eta(z)| + \left|\sum_{\eta =0}^{\eta_o} h_\eta  - f(z) \right|
  \end{equation*}
  \begin{eqnarray*}
     &\leq&  \epsilon/2^{\eta_o+1} + \left| h_{\eta_o}(z) - \rho_{\eta_o}(z) \left( f(z) - \sum_{\eta=0}^{\eta_o-1}h_\eta(z)\right)\right| 
      + (1-\rho_{\eta_o}(z))\left|\sum_{\eta =0}^{\eta_o - 1 } h_\eta(z)  - f(z) \right| \\
     &\leq&  \epsilon(z)/2^{\eta_o+1} + \epsilon(z)/2^{\eta_o+1} 
      + (1-\rho_{\eta_o}(z))\left|h_{\eta_o -1} - \rho_{\eta_o -1}(z)\left(f(z) - \sum_{\eta =0}^{\eta_o - 2 } h_\eta(z)   \right) \right| \\
     &\leq&  \epsilon(z)/2^{\eta_o+1} + \epsilon(z)/2^{\eta_o+1} + \epsilon(z)/2^{\eta_o} \leq \epsilon(z)
  \end{eqnarray*}
    \end{proof}

\section{The union of two Lipschitz graphs in $\cplx^2$}
\label{section:union-of-lipschitz-graphs}

\subsection{Polynomial convexity of the union of two Lipschitz graphs}

\label{subsection:poly-cnx-union-of-lipschitz-graphs}

In this section, 
we will study the  polynomial convexity of $\L_1 \cup \L_2$
where $\L_1$ and $\L_2$ are Lipschitz graphs as given in \Cref{twoLipschitz}.

We will start by showing the polynomial convexity of the union of two Lipschitz graphs in $\C^2$.
The proof of this theorem generalizes the calculation done in \cite[Theorem 1.2]{gorai2011local}.

\begin{proof}[Proof of \Cref{prop:poly-cnv-graphs}]
  The matrix $A$ is $2 \times 2$ real valued. Then $A$ is similar to one of the following real matrix:
  \begin{eqnarray}
    \begin{bmatrix}
      \lambda & 1 \\
      0 & \lambda 
    \end{bmatrix} & \lambda \in \R  \label{eqn:same-eigen-block-form}, \\
    \begin{bmatrix}
      \lambda_1 & 0 \\
      0 & \lambda_2
    \end{bmatrix} & \lambda_1, \lambda_2 \in \R \label{eqn:diag-form}, \\
    \begin{bmatrix}
      s & -t \\
      t & s
    \end{bmatrix} & s, t \in \R \label{eqn:img-root-form},
  \end{eqnarray} 

This can be showed by looking at the different possible Jordan forms
for a $2\times 2$ matrix with real entries. If the roots for the
characteristic polynomial are imaginary then we obtain (\ref{eqn:img-root-form}).
If the roots are real and $A$ is not diagonizable, then we obtain (\ref{eqn:same-eigen-block-form}).
Finally, if the roots are real and $A$ is diagonizable, then we obtain (\ref{eqn:diag-form}).

There is an invertible complex $2 \times 2$ matrix such that $SAS^{-1}$ has one
of the above forms. First, we look at what happens to $(A+iI)\R^n$ under the action of
$S$. Say $x\in \R^n$, then $S(A+iI)x = SAS^{-1}y+iSIS^{-1}y = (SAS^{-1}+iI)y$, where $y=S^{-1}x$.
Which gives $(SAS^{-1}+iI)\R^n = S((A+iI)\R^n)$.
We will study the action of $S$ on $\L_1$ and $\L_2$:
\begin{eqnarray*}
  S\L_1 &=& \{ S(x+ i\psi_1(x)): x\in\R^n \} \\
  &=& \{ S(x) + iS(\psi_1(x)): x\in\R^n \} \\
  &=& \{ y + iS(\psi_1(S^{-1}y)): y\in\R^n \}.
\end{eqnarray*}
Similarly for $\L_2$:
\begin{eqnarray*}
  S\L_2 &=& \{ S((A+iI)(x+ i\psi_2(x))): x\in\R^n \} \\
  &=& \{ S((A+iI)(S^{-1}y+ i\psi_2(S^{-1}y))): y\in\R^n \} \\
  &=& \{ (SAS^{-1}+iI)y+ i(SAS^{-1}+iI)S\psi_2(S^{-1}y): y\in\R^n \}.
\end{eqnarray*}
If we set $A' = SAS^{-1}$, $\psi_j'(x)=S\psi_j(S^{-1}x)$, and $\alpha_j' =\|S\|\|S^{-1}\|\alpha_j$. Then the above has the form
\begin{eqnarray*}
  \L_1' &=& \{ (x+ i\psi_1'(x))): x\in\R^n \}, \\
  \L_2' &=& \{ (A'+iI)(x+ i\psi_2'(x))): x\in\R^n \}.
\end{eqnarray*}
The Lipschitz condition will still be followed with the new constants $\alpha_j'$
\begin{eqnarray*}
  |\psi_j'(x) - \psi_j'(y)| &=& |S\psi_j(S^{-1}x) -  S\psi_j(S^{-1}y)| \\
  &\leq& \|S\| |\psi_j(S^{-1}x) -  \psi_j(S^{-1}y)|  \\
  &\leq& \|S\| \alpha_j |S^{-1}x - S^{-1}y|  \\
  &\leq& \alpha_j \|S\|\|S^{-1}\| |x - y| 
  \leq \alpha_j' |x - y|.
\end{eqnarray*}
Hence, we can consider $A$ to be one of the three forms given above.
We will work cases for each of form: (\ref{eqn:same-eigen-block-form}), (\ref{eqn:diag-form}),
and (\ref{eqn:img-root-form}).

\textbf{Case 1:} Let $A$ be of type (\ref{eqn:same-eigen-block-form}).

Let $\lambda \in \R$
$$
A = \begin{bmatrix}
  \lambda & 1 \\
  0 & \lambda
\end{bmatrix}.
$$
We will show any compact set $K \subset \L_1 \cup \L_2$ is polynomially convex. We can always decompose $K = K_1 \cup K_2$
for compacts sets $K_1 \subset \L_1$ and $K_2 \subset \L_2$.

Define $G:\C^2 \times \C^2 \to \C$ as
$$
G(z,w) = z_1w_1 + z_2 w_2.
$$
Consider the polynomial $p:\C^2 \to \C$ given by
$$
p(z) := G((A-iI)z,z).
$$
More explicitly, if $z=(z_1, z_2) \in \C^2$ then $p(z)$ simplifies to
$$
p(z) = z_{1} \left(z_{1} \left(\lambda - i\right) + z_{2}\right) + z_{2}^{2} \left(\lambda - i\right).
$$
We will show
\begin{eqnarray}
p(K_1) \subset \{z\in \mathbb{C}: \text{Im}(z)<0\} \cup \{0\} \label{eqn:type-1-phi-1-aim},\\
p(K_2)\subset \{z\in \mathbb{C}: \text{Im}(z)>0\} \cup \{0\} \label{eqn:type-1-phi-2-aim},
\end{eqnarray}
for an appropriate choice of $\alpha_1>0$ and $\alpha_2>0$.

Let $z\in \L_1$. Then $z=(x,y) + i\psi_1(x,y)$ for some $(x,y)\in \R^2$.
Solving for $\text{Im}(p(z))$, we get
\begin{dmath}
$$
\text{Im}(p(z)) = 
 \left(\psi_1^{(1)}(x,y)\right)^{2} 
 + \left(\psi_1^{(2)}(x,y)\right)^{2} 
 + 2 \psi_1^{(1)}(x,y) \lambda x 
 + 2 \psi_1^{(2)}(x,y) \lambda y 
 + \psi_1^{(1)}(x,y) y 
 + \psi_1^{(2)}(x,y) x
 - x^{2} - y^{2}
  $$ \label{eqn:type-1-phi-1},
\end{dmath}
where $\psi_1(x,y) = (\psi_1^{(1)}(x,y), \psi_1^{(2)}(x,y))$.

By the Lipschitz condition on $\psi_1$, we have
\begin{eqnarray*}
  \|\psi_1((x,y)) - \psi_1((0,0))\| &\leq& \alpha_1 \|(x,y) - (0,0)\|, \\
  \|\psi_1((x,y))\| &\leq& \alpha_1 \|(x,y)\|.
\end{eqnarray*}
Set $v=\|(x,y)\|$. Then $(\psi_1^{(1)}(x,y))^2+(\psi_1^{(2)}(x,y))^2 \leq \alpha_1^2 v^2$.
Further, $|x| \leq v$, $|y| \leq v$, $|\psi_1^{(1)}(x,y)| \leq \|\psi_1(x,y)\| \leq \alpha_1 v$,
and $|\psi_1^{(2)}(x,y)| \leq \|\psi_1(x,y)\| \leq \alpha_1 v$.
Now we make use of inequality (\ref{eqn:type-1-phi-1}). 
$$
\text{Im}(p(z)) \leq v^2 (\alpha_1^2 + 4 \alpha_1 \lambda + 2 \alpha_1 - 1).
$$
Choose $\alpha_1$ small enough such that: 
$$ \alpha_1^2 + 4 |\lambda| \alpha_1 + 2 \alpha_1 < 1.$$
Then $\text{Im}(p(z)) < 0 $ for $z \neq 0 \in \L_1$.

Now we look at the other Lipschitz graph. Let $z\in \L_2$. Then $z=(A+iI)(x,y) + i(A+iI)\psi_2(x,y)$ for some $(x,y)\in \R^2$.
If we solve for the explicit form of $z$, we get
$$
z=\left[\begin{matrix}i \psi^{(2)}(x,y) + y + \left(\lambda + i\right) \left(i \psi^{(1)}(x,y) + x\right)\\\left(\lambda + i\right) \left(i \psi^{(2)}(x,y) + y\right)\end{matrix}\right].
$$

Now solving for $\text{Im}(p(z))$:
\begin{dmath}
$$
\text{Im}(p(z)) =
- \left(\psi_2^{(1)}(x,y)\right)^{2} (1+\lambda^{2})
- \left(\psi_2^{(2)}(x,y)\right)^{2} (1+\lambda^{2})
- 2 \psi_2^{(1)}(x,y) \psi_2^{(2)}(x,y) \lambda 
+ 2 \psi_2^{(1)}(x,y) \lambda^{3} x 
+ 3 \psi_2^{(1)}(x,y) \lambda^{2} y 
+ 2 \psi_2^{(1)}(x,y) \lambda x 
+ \psi_2^{(1)}(x,y) y 
+ 2 \psi_2^{(2)}(x,y) \lambda^{3} y 
+ 3 \psi_2^{(2)}(x,y) \lambda^{2} x 
+ 6 \psi_2^{(2)}(x,y) \lambda y 
+ \psi_2^{(2)}(x,y) x 
+ \lambda^{2} x^{2} 
+ \lambda^{2} y^{2} 
+ 2 \lambda x y 
+ x^{2} 
+ y^{2}.
$$
\end{dmath}
Again, we obtain similar inequalities.
Set $v=\|(x,y)\|$. Then $(\psi_2^{(1)}(x,y))^2+(\psi_2^{(2)}(x,y))^2 \leq \alpha_2^2 v^2$.
Further, $|x| \leq v$, $|y| \leq v$, $|\psi_2^{(1)}(x,y)| \leq \|\psi_2(x,y)\| \leq \alpha_2 v$,
and $|\psi_2^{(2)}(x,y)| \leq \|\psi_2(x,y)\| \leq \alpha_2 v$.
$$
\text{Im}(p(z)) \geq
- [\alpha_2^2(|\lambda| + 1)^2 + \alpha_2 (4 |\lambda|^3 + 6 |\lambda|^2 + 10 |\lambda| + 2)] v^2
+ \lambda^2 x^2 + y^2 + (\lambda y + x)^2.
$$
If $\lambda > 1$, then $\lambda^2 x^2 + y^2  = v^2 + (\lambda^2 - 1) y^2$. Then
$$
\text{Im}(p(z)) \geq
- [\alpha_2^2(|\lambda| + 1)^2 + \alpha_2 (4 |\lambda|^3 + 6 |\lambda|^2 + 10 |\lambda| + 2) - 1] v^2
+ (\lambda^2 -1)y^2 + (\lambda y + x)^2.
$$

Take $\alpha_2 > 0$ small enough so that
$$
\alpha_2^2(|\lambda| + 1)^2 + \alpha_2 (4 |\lambda|^3 + 6 |\lambda|^2 + 10 |\lambda| + 2) < 1.
$$

If $\lambda < 1$, then 
$
\lambda^2 x^2 + y^2
= \lambda^2 v^2 + (1 - \lambda^2) y^2
$.
Then
$$
\text{Im}(p(z)) \geq
- [\alpha_2^2(|\lambda| + 1)^2 + \alpha_2 (4 |\lambda|^3 + 6 |\lambda|^2 + 10 |\lambda| + 2) - \lambda^2] v^2
+ (1 - \lambda^2)y^2 + (\lambda y + x)^2.
$$

Take $\alpha_2 > 0$ small enough so that
$$
\alpha_2^2(|\lambda| + 1)^2 + \alpha_2 (4 |\lambda|^3 + 6 |\lambda|^2 + 10 |\lambda| + 2) < \lambda^2.
$$

This implies $\text{Im}(p(z)) > 0$ when $z \neq 0 \in \L_2$.

Finally, by both these inequalities we can decipher $p^{-1}(0) \cap (K_1 \cup K_2) = \{0\}$, which is clearly polynomially convex. 
Hence, $K_1 \cup K_2$ is polynomially convex by Kallin's \Cref{lemma:Kallins-poly}.

\textbf{Case 2:} Let $A$ be of type (\ref{eqn:diag-form}). \\
Let $\lambda_1, \lambda_2 \in \R$ and
$$
A = \begin{bmatrix}
  \lambda_1 & 0 \\
  0 & \lambda_2
\end{bmatrix}.
$$
We will mimic the same steps as in Case 1 with the same polynomial $p$.
Pick arbitrary compact sets $K_1 \subset \L_1$ and $K_2 \subset \L_2$. Then we will show
\begin{eqnarray*}
p(K_1) \subset \{z\in \mathbb{C}: \text{Im}(z)<0\} \cup \{0\}, \\
p(K_2)\subset \{z\in \mathbb{C}: \text{Im}(z)>0\} \cup \{0\}.
\end{eqnarray*}
Let $z\in \L_1$. Then $z=(x,y) + i\psi_1(x,y)$ for some $(x,y)\in \R^2$.
Solving for $\text{Im}(p(z))$, we get
\begin{dmath}
$$
\text{Im}(p(z)) = 
\left(\psi_1^{(1)}(x,y)\right)^{2} 
+ \left(\psi_1^{(2)}(x,y)\right)^{2} 
+ 2 \psi_1^{(1)}(x,y) \lambda_1 x 
+ 2 \psi_1^{(2)}(x,y) \lambda_2 y
- x^{2} - y^{2}
$$
\end{dmath}
The same inequality shown in case 1 works here.
$$
\text{Im}(p(z)) \leq (\alpha_1^2 + 2 \alpha_1 |\lambda_1| + 2 \alpha_1 |\lambda_2| - 1) v^2.
$$
Choose $\alpha_1>0$ small enough such that
$$
\alpha_1^2 + 2 \alpha_1 |\lambda_1| + 2 \alpha_1 |\lambda_2| < 1.
$$
Then $\text{Im}(p(z)) < 0 $ for $z \neq 0 \in \L_1$.

Now let $z\in \L_2$, say $z=(A+iI)(x,y) + i(A+iI)\psi_2(x,y)$ for some $(x,y)\in \R^2$.
If we solve for the explicit form of $z$, we get
$$
z=\left[\begin{matrix}\left(\lambda_{1} + i\right) \left(i \psi_2^{(1)}(x,y) + x\right)\\\left(\lambda_{2} + i\right) \left(i \psi_2^{(2)}(x,y) + y\right)\end{matrix}\right].
$$
\begin{dmath*}
$$
\text{Im}(p(z)) = 
- \left(\psi_2^{(1)}(x,y)\right)^{2} \lambda_{1}^{2} 
- \left(\psi_2^{(1)}(x,y)\right)^{2} 
+ 2 \psi_2^{(1)}(x,y) \lambda_{1}^{3} x 
+ 2 \psi_2^{(1)}(x,y) \lambda_{1} x 
- \left(\psi_2^{(2)}(x,y)\right)^{2} \lambda_{2}^{2} 
- \left(\psi_2^{(2)}(x,y)\right)^{2} 
+ 2 \psi_2^{(2)}(x,y) \lambda_{2}^{3} y 
+ 2 \psi_2^{(2)}(x,y) \lambda_{2} y 
+ \lambda_{1}^{2} x^{2} 
+ \lambda_{2}^{2} y^{2} 
+ x^{2} 
+ y^{2}.
$$
\end{dmath*}
Combining this with the inequalities, we get
$$
\text{Im}(p(z)) \geq - [
  \alpha^2(1+|\lambda_1|^2 + |\lambda_2|^2) + 2\alpha(|\lambda_1|^3 + |\lambda_2|^3 + |\lambda_1| + |\lambda_2|)
] v^2 + v^2 + \lambda_1^2 x^2 + \lambda_2^2 y^2.
$$
If we choose $\alpha >0$ small enough such that
$$
\alpha^2(1+|\lambda_1|^2 + |\lambda_2|^2) + 2\alpha(|\lambda_1|^3 + |\lambda_2|^3 + |\lambda_1| + |\lambda_2|) < 1
$$
Then $\text{Im}(p(z)) >0$ when $z\neq 0 \in \L_2$.

Again, with both these results we get $p^{-1}(0) \cap (K_1 \cup K_2) = \{0\}$ and this is polynomially convex.
Hence, $K_1 \cup K_2$ is polynomially convex by Kallin's \Cref{lemma:Kallins-poly}.

\textbf{Case 3:} Let $A$ be of type (\ref{eqn:img-root-form}). \\
Let $s, t \in \R$ and
$$
A = \begin{bmatrix}
  s & -t \\
  t & s
\end{bmatrix}.
$$

Take compacts $K_1 \subset \L_1$ and $K_2 \subset \L_2$. We wish to show $K_1\cup K_2$ is polynomially convex.
Consider the polynomial:
$$
p(z) = z_1^2 + z_2^2.
$$

Let $z\in \L_1$. Then $z=(x,y) + i\psi_1(x,y)$ for some $(x,y)\in \R^2$.
Solving for $\text{Im}(p(z))$, we get
\begin{dmath*}
$$
\text{Im}(p(z)) = 2 \psi_1^{(1)}(x,y) x + 2 \psi_1^{(2)}(x,y) y.
$$
\end{dmath*}
And solving for $\text{Re}(p(z))$, we get
\begin{dmath*}
$$
\text{Re}(p(z)) = - \left(\psi_1^{(1)}(x,y)\right)^{2} - \left(\psi_1^{(2)}(x,y)\right)^{2} + x^{2} + y^{2}.
$$
\end{dmath*}
Using the inequalities here, we get
$
\text{Re}(p(z)) \geq (1-\alpha_1^2)v^2.
$
If we choose $\alpha_1$ so that $ \alpha_1^2 < 1$. Which is true for Lipschitz graphs with $\alpha_1 < 1$.
Hence, $\text{Re}(p(z)) >0$ when $z\neq 0 \in \L_1$.

Let $z\in \L_2$, say $z=(A+iI)(x,y) + i(A+iI)\psi_2(x,y)$ for some $(x,y)\in \R^2$.
If we solve for the explicit form of $z$, we get
$$
z=\left[\begin{matrix}- t \left(i \psi_2^{(2)}(x,y) + y\right) + \left(s + i\right) \left(i \psi_2^{(1)}(x,y) + x\right)\\t \left(i \psi_2^{(1)}(x,y) + x\right) + \left(s + i\right) \left(i \psi_2^{(2)}(x,y) + y\right)\end{matrix}\right].
$$
Now we solve for $\text{Im}(p(z))$ and $\text{Re}(p(z))$:
\begin{dmath*}
$$
\text{Im}(p(z)) = 
- 2 \left(\psi_2^{(1)}(x,y)\right)^{2} s 
+ 2 \psi_2^{(1)}(x,y) s^{2} x 
+ 2 \psi_2^{(1)}(x,y) t^{2} x 
- 2 \psi_2^{(1)}(x,y) x 
- 2 \left(\psi_2^{(2)}(x,y)\right)^{2} s 
+ 2 \psi_2^{(2)}(x,y) s^{2} y 
+ 2 \psi_2^{(2)}(x,y) t^{2} y 
- 2 \psi_2^{(2)}(x,y) y 
+ 2 s x^{2} 
+ 2 s y^{2}.
$$
\end{dmath*}
\begin{dmath*}
$$
\text{Re}(p(z)) = 
- \left(\psi_2^{(1)}(x,y)\right)^{2} s^{2} 
- \left(\psi_2^{(1)}(x,y)\right)^{2} t^{2} 
+ \left(\psi_2^{(1)}(x,y)\right)^{2} 
- 4 \psi_2^{(1)}(x,y) s x 
- \left(\psi_2^{(2)}(x,y)\right)^{2} s^{2} 
- \left(\psi_2^{(2)}(x,y)\right)^{2} t^{2} 
+ \left(\psi_2^{(2)}(x,y)\right)^{2} 
- 4 \psi_2^{(2)}(x,y) s y 
+ (t^{2} + s^{2} - 1) x^{2} 
+ (t^{2} + s^{2} - 1) y^{2}.
$$
\end{dmath*}
We consider two sub-cases based on if $t^2 + s^2 < 1$.

\textbf{Case 3.1} Say $t^2 + s^2 < 1$.
Using the inequalities on $\text{Re}(p(z))$, we get
$$
\text{Re}(p(z)) \leq [(t^2+s^2-1)(1-\alpha_2^2) - 8\alpha_2 |s|] v^2.
$$
If choose $\alpha_2 >0 $ so small that
$$
(t^2+s^2-1)(1-\alpha_2^2) < 8\alpha_2 |s|.
$$
Then $\text{Re}(p(z)) <0$ when $z\neq 0 \in \L_2$.

Again it is clear that $p^{-1}(0) \cap (K_1 \cup K_2) = \{0\}$.
By Kallin's \Cref{lemma:Kallins-poly}, we get $K_1 \cup K_2$ is polynomially convex.

\textbf{Case 3.2} Say $t^2 + s^2 \geq 1$.
We will show that $p(K_1)$ and $p(K_2)$ lie in different angular sectors which only
intersect at the origin. Then we claim $K_1 \cup K_2$ is polynomially convex by
Kallin's \Cref{lemma:Kallins-poly}.

Consider the following angular sectors
\begin{eqnarray*}
  V_1 &=& \{x+iy \in \C: |y| \leq \epsilon |x| \}, \\
  V_2 &=& \{x+iy \in \C: |(t^2+s^2-1)y - 2sx| \leq \epsilon |y| \}, \\
\end{eqnarray*}
where $\epsilon > 0$ small enough such that $V_1 \cup V_2 = \{0\}$.

For suitable choice of  $\alpha_1$ and $\alpha_2$, we claim:
$p(\L_1) \subset V_1$ and 
$p(\L_2) \subset V_2$.

\begin{itemize}
  \item
    Let $z \in \L_1$. Apply the inequalities on $|\text{Im}(p(z))|$:
    $$
    |\text{Im}(p(z))| \leq  2\alpha_1 v^2,
    $$
    and on $|re(p(z))|$
    $$
    |\text{Re}(p(z))| \geq (1-\alpha_1^2) v^2.
    $$
    If we choose $\alpha_1$ small enough such that such that $2\alpha_1 < \epsilon(1-\alpha_1^2)$, then
    $$
    |\text{Im}(p(z))| \leq \epsilon |\text{Re}(p(z))|.
    $$
    Hence, $p(\L_1) \subset V_1$.

  \item
    Let $z \in \L_2$. Calculating $(t^2+s^2-1)(re(p(z))) - 2s (\text{Im}(p(z))) = $ 
    \begin{dmath*}
    $$
    2 \psi^{(1)}(x,y) s^{4} x + 4 \psi^{(1)}(x,y) s^{2} t^{2} x + 4 \psi^{(1)}(x,y) s^{2} x + 2 \psi^{(1)}(x,y) t^{4} x - 4 \psi^{(1)}(x,y) t^{2} x + 2 \psi^{(1)}(x,y) x + 2 \psi^{(2)}(x,y) s^{4} y + 4 \psi^{(2)}(x,y) s^{2} t^{2} y + 4 \psi^{(2)}(x,y) s^{2} y + 2 \psi^{(2)}(x,y) t^{4} y - 4 \psi^{(2)}(x,y) t^{2} y + 2 \psi^{(2)}(x,y) y .
    $$
    \end{dmath*}
    Using inequalities, we get
    $$
    |(t^2+s^2-1)(\text{Re}(p(z))) - 2s (\text{Im}(p(z)))| \leq \alpha_2 (4s^4+ 8s^2 t^2 + 8 s^2 + 4 t^4 + 8 t^2) v^2.
    $$
    Now we will find a lower bound for $|\text{Im}(p(z))|$. First note that $|s|\neq 0$;
    to ensure $A$ has no eigenvalue which is purely imaginary of modulus greater than 1, we must have
    $|s| \neq 0$. Now we use the inequalities on $|\text{Im}(p(z))|$: 
    $$
    |\text{Im}(p(z))| \geq 2|s| v^2 - \alpha_2[2\alpha_2 |s| + 4s^2+t^2 + 4]v^2.
    $$
    We can choose $\alpha_2>0$ small such that
    \begin{eqnarray*}
      \epsilon(2|s| - \alpha_2[2\alpha_2 |s| + 4s^2+t^2 + 4]) \geq  \alpha_2 (4s^4+ 8s^2 t^2 + 8 s^2 + 4 t^4 + 8 t^2)
    \end{eqnarray*}
    This ensures
    $$
    |(t^2+s^2-1)(\text{Re}(p(z))) - 2s (\text{Im}(p(z)))| \leq \epsilon |\text{Im}(p(z))|.
    $$
    Hence, $p(\L_2) \subset V_2$.
\end{itemize}

Again, $p^{-1}(0) \cap (K_1 \cup K_2) = \{0\}$: The angular sectors intersected with a closed ball is polynomially convex. And $A_1 \cap A_2 = \{0\}$.
Hence $K_1 \cup K_2$ is polynomially convex by Kallin's \Cref{lemma:Kallins-poly}.
\end{proof}

\subsection{Carleman approximation on the union of two Lipschitz graphs}

The proof of \Cref{thm:lipschitz-main} will follow similar arguments to that of the union of two totally real
subspaces \cite{manne1994carlemantwo}. Before this, we show the following
proposition:

\begin{proposition} \label{prop:lipschitz-sqeeze}
  Let $\L_1$ and $\L_2$ be as in \Cref{thm:lipschitz-main}. Fix $R > 1$, and let
  $$
  K =  \{z\in \L_1 \cup \L_2 : |z| \leq R\} \text{ and }
  K_j =\{z\in  \L_j : 1 \leq |z| \leq R\} \quad j = 1, 2.
  $$

  Then there exists $\delta > 0$ such that
  for any continuous function $f \in \smoo(\L_1 \cup \L_2) $ with compact support
  in $K_1\cup K_2$ and 
  any $\epsilon > 0$ there is an entire function $h \in \holo(\C^n)$ such that
  $\|h-f\|_{K} < \epsilon $ and $|h(z)| < \epsilon$ for $|z|\leq \delta$.
\end{proposition}

\begin{proof}
  We set $\phi_1(z) = z$ and $\phi_2(z) = (A+iI)^{-1}(z)$. Then
  $$
    \phi_1(\L_1) = \{x+i\psi_1(x): x\in \R^n\}
    \quad\&\quad
    \phi_2(\L_2) = \{x+i\psi_2(x): x\in \R^n\}.
  $$
  Let
  $$
  \tilde{L}_1 =     \{x+iy \in \C^n: \|y\|\leq \alpha_1 \|x\| \} \quad\&\quad
  \tilde{L}_2 =     (A+iI)\{x+iy \in \C^n: \|y\|\leq \alpha_2 \|x\| \}, \\
  $$
  $$
  \tilde{K}_1 =     \tilde{L}_1 \cap \{z\in \C^n: 1 \leq |z| \leq R \} \quad \& \quad
  \tilde{K}_2 =     \tilde{L}_2 \cap \{z\in \C^n: 1 \leq |z| \leq R \}.
  $$
  Then by (\ref{eqn:plane-angular}) we can choose $\gamma >0$ such that
  \begin{eqnarray*}
    d(\phi_j(\tilde K_j), \phi_j(\tilde L_k)) > 3 \gamma
  \end{eqnarray*}
  for  $i, j = 1,2$ where $i\neq j$.
  As $K_j \subset \tilde K_j$ and $L_k \subset \tilde L_k$:
  $$
    d(\phi_j(K_j), \phi_j(L_k)) > 3 \gamma
  $$

  We define a neighborhood of $K$ given by
  $
    \Omega_o = \{ z\in \C^n : d(z,K)<\eta \}
  $.
  where $\eta > 0$ is taken so small that
  $
  \phi_j(\Omega_o) \subset \{z\in \C^n : d(z, \phi_j(K)) < \gamma \}
  $.
  for $j=1,2$.
  
  \newcommand{\U}{\tilde{U}_j}
  \newcommand{\V}{\tilde{V}_j}
  \newcommand{\W}{\tilde{W}_j}
  For $j = 1,2$, define open sets $\U$ and $\V$  which cover $\phi_j(K)$:
  $$
  \U = \left\{ z \in \C^n: d\left(z,\phi_j(K)\right) < \gamma \text{ and } d\left(z, \phi_j(K_j)\right) < 2\gamma \right\},
  $$
  $$
  \V = \left\{ z \in \C^n: d\left(z,\phi_j(K)\right) < \gamma \text{ and } d\left(z, \phi_j(K_j)\right) > \frac{3}{2}\gamma \right\}.
  $$
  Observe that, $\phi_j(K_j) \subset \U$, $\phi_j(L_k\cap K) \subset \V$ for $k\neq j$ and
  $0\in \V$.

  Define the convolution as per \Cref{prop:lipshitz-one-approx}:
  \newcommand{\h}{\tilde{h}_j^{(t)}}
  \renewcommand{\a}{\tilde{\alpha}_j^{(t)}}
  \renewcommand{\b}{\tilde{\beta}_j^{(t)}}
  $$
  \h(z)  :=  \left(\frac{1}{t\sqrt{\pi}}\right)^n \int_{\L} f\circ \phi_j^{-1}(u) e^{-\frac{(z-u)^2}{t^2}}du.
  $$
  All limits are considered as $t\to 0^+$.
  Now by \Cref{prop:lipshitz-one-approx} we have $\h \in \holo(\C^n)$,
  and $\h \to f\circ \phi_j^{-1}$ uniformly  on $\phi_j(L_j)$.
  Let $\W = \V \cap \U$, we will show $\W \subset A_{\frac{\gamma^2}{4}}$; this gives us
  $\h \to 0$ uniformly on $\W$.

  Let $z=(x_1+i y_1, ..., x_n+i y_n) \in \W$ and $\eta = (u_1+i\psi_j^{(1)}(u), ..., u_n + i\psi_j^{(n)}(u)) \in \phi_j(K_j)$.
  \begin{eqnarray*}
    Re(z-\eta)^2 &=& Re \sum_{l=1}^n (z_l - \eta_l)^2, \\
    &=& \sum_{l=1}^n (x_l - u_l)^2 - (y_l - \psi_j^{(l)}(u))^2, \\
    &=& \sum_{l=1}^n (x_l - u_l)^2 - \sum_{l=1}^n (y_l - \psi_j^{(l)}(u))^2.
  \end{eqnarray*}
  As $z\in\W \subset \V$, $d(z,u) > \frac{3}{2}\gamma$ which gives
  $$
    \sum_{l=1}^n (x_l - u_l)^2 + \sum_{l=1}^n (y_l - \psi_j^{(l)}(u) )^2 > \frac{9}{4}\gamma^2
  $$
  Note that $d(z,\psi_j(K_j)) < 2\gamma$, $d(\psi_j(K_j), L_k) > 3\gamma$ and
  $d(z, \psi_j(\L_1 \cup \L_2)) \leq d(z,\psi_j(K))<\gamma$  implies $d(z,\psi_j(\L_j)) < \gamma$.
  With this we get $d(z,\psi_j(\L_j)) = \|y-\psi(x)\|_{\R^n} < \gamma$ as the minima is attained
  when we look at the distance between  $z$ and $x+i\psi(x) \in \L_j$.
  \begin{eqnarray*}
    \sum_{l=1}^n (y_l - \psi_j^{(l)}(u) )^2  &=& \| y - \psi_j(u) \|^2_{\R^n} \\
    &\leq& \| y - \psi_j(x) \|^2_{\R^n} + \| \psi_j(x) - \psi_j(y) \|^2_{\R^n} \\
    &\leq& \gamma^2 + \alpha_j^2 \|x - u \|_{\R^n}^2.
  \end{eqnarray*}
  Combining the last two inequalities, 
  \begin{eqnarray*}
  \|x - u\|_{\R^n}^2 > \frac{5}{4}\gamma^2 - \alpha_j^2 \|x\|_{\R^n}^2, \\
    (1+\alpha_j^2)\|x - u\|_{\R^n}^2 > \frac{5}{4}\gamma^2, \\
    \|x - u\|_{\R^n}^2 > (1+\alpha_j^2)^{-1}\frac{5}{4}\gamma^2.
  \end{eqnarray*}
  Finally,
  \begin{eqnarray*}
    Re(z-u)^2 &= \|x-u\|^2 + \|y-\psi_j(u)\|^2 
    &\geq \|x-u\|^2 - \gamma^2 - \alpha_j^2 \|x - u \|^2 \\
    &\geq (1-\alpha_j^2)\|x-u\|^2 - \gamma^2 
    &\geq (1-\alpha_j^2)(1+\alpha_j^2)^{-1}\frac{5}{4}\gamma^2 - \gamma^2 \\
    &\geq (1-9\alpha_j^2)(1+\alpha_j^2)^{-1} \gamma^2. &
  \end{eqnarray*}
  If $\alpha_j^2 < 1/3$ then 
  $\W \subset A_{c\gamma}$ where $c =(1-9\alpha_j^2)(1+\alpha_j^2)^{-1}$ 
  and $\h \to 0$ uniformly on $\W$ (\Cref{prop:lipshitz-one-approx}).

  The set $K$ is polynomially convex as it is a compact subset of $\cup \L_j$.
  We can find a pseudoconvex domain $\Omega$ such that $K \subset \Omega \subset \Omega_o$
  by a polynomial polyhedra.

  Let $0 \leq j \leq m$.
  Set $U_j = \Omega \cap \U$, 
  $V_j = \Omega \cap \V$ and
  $W_j = V_j \cap U_j = \Omega \cap \phi_j^{-1}(\W)$.
  Then $U_j \cup V_j = \Omega$ and $\h \circ \phi_j \to 0$ uniformly on $W_j$.

  Consider the linear operator $T: \holo(U_j) \bigoplus \holo(V_j) \to \holo(W_j)$
  given by
  $$T(\alpha, \beta) = (\alpha - \beta)|_{W_j}$$
  The Cousin I problem can be solved on the pseudoconvex set $\Omega$ with $\{U_j, V_j\}$;
  Hence $T$ is surjective.
  This allows us to use the open mapping theorem for Fréchet spaces.
  As $\h \circ \phi_j \to 0$ in $\holo(W_j)$, we can solve for $(\a, \b) \in \holo(U_j) \bigoplus \holo(V_j)$
  with
  $$T(\a,\b) = (\a-\b)|_{W_j} = \h \circ \phi_j$$
  such that $(\a, \b) \to 0$, i.e., $\a \to 0$ in $\holo(U_j)$ and $\b \to 0$ in $\holo(V_j)$.

  \newcommand{\hjt}{h_j^{(t)}}
  \vspace{1em}
  We have seen that $0 \in V_j$ for all $j$.
  Take $\delta > 0$ such that the closed ball $B(0,\delta) \subset\subset V_j$ for all $j$.
  Define $\hjt \in \holo(\Omega)$ with
  $\hjt = \h \circ \phi_j -\a$ on $U_j$ and $\hjt = -\b$ on $V_j$. This is well defined
  due to the Cousins I problem.

  Then $\hjt \to 0$ in $\holo(V_j)$; this gives: 
  $\hjt \to 0$ uniformly on $\L_k \cap K  \subset\subset V_j$ for $k\neq j$,
  and $\hjt \to 0$ uniformly on $B(0,\delta)  \subset \subset V_j$ for $k\neq j$.
  Further,
  $\hjt \to f$ uniformly on compact subsets of
  $\L_j \cap V_j$ as $f=0$ on $\L_j - K_j$.

  The entire function $\h \circ \phi \to f$ uniformly on $\L_j$ and
  $\a \to 0$ in $\holo(U_j)$ which implies
  $\hjt \to f$ uniformly on compact subsets of $\L_j \cap U_j$.
  Combining with the previous argument, $\hjt \to f$ uniformly on $\L_j \cap K$,
  $\hjt \to 0$ uniformly on $\L_k \cap K$ for $k\neq j$,
  and $\hjt \to 0$ uniformly on $B(0,\delta)$.

  \renewcommand{\ht}{h^{(t)}}
  Define $\ht := \sum_{j=1}^{2} \hjt$.
  Then $\ht \to f$ uniformly on $\L_j \cap K$ for all $j$.
  Hence, $\ht \to f$ uniformly on $K$ and $\ht \to 0$ uniformly on $B(0,\delta)$.

\end{proof}

\begin{corollary} \label{cor:vanish-lipshitz}
  For $R_o>1$, we can choose $\delta > 0$ for which the following holds:
  If $f_o \in \smoo(\L_1 \cup \L_2)$ such that $f_o(z)=0$ for $|z| \leq r$ for some $r>0$.
  Then
  for any $\epsilon > 0$,
  there exists $h_o \in \holo(\C^n)$ such that $|h_o(z)-f_o(z)|<\epsilon$
  for $z \in \left(\L_1 \cup \L_2 \right) \cap  \{z\in \C^n: |z| \leq rR_o\} $, and
  $|h_o(z)|<\epsilon$ for $|z|\leq r\delta$.
\end{corollary}

\begin{proof} 
  We define
  $$f(z) := f_o(z/r) \rho_{R_o}(|z|)$$
  where $\rho_{R_o}: [0,+\infty) \to [0,1]$ is a continuous function that is
  $1$ on $[0,R_o]$ and $0$ on $[R_o+1, + \infty)$. Let $R=R_o+1$;
  it follows that $\text{supp}(f_o) \subset \cup K_j$ with the notations of
  \Cref{prop:lipschitz-sqeeze}. By the proposition, we have $\delta>0$ corresponding
  to $R$. Then there is an entire function $h\in \holo(\C^n)$ such that
  $\|f - h\|_K < \epsilon $ and $\|h\|_{B(0,\delta)} < \epsilon$.
  Define entire function $h_o(z) := h(rz)$. We have $|f(z) - h(z)| < \epsilon$ for
  $z \in (\L_1 \cup \L_2) \cap \{z\in\C^n: |z| \leq R\}$.
  Replacing $f$ with $f_o$, 
  $|f_o(z/r) - h_o(z/r)| < \epsilon$ for
  $z \in (\L_1 \cup \L_2) \cap \{z\in\C^n: |z| \leq R_o\}$.
  Which is 
  $|f_o(z) - h_o(z)| < \epsilon$ for
  $z \in (\L_1 \cup \L_2) \cap \{z\in\C^n: |z| \leq rR_o\}$ and 
  $|h_o(z)| < \epsilon$ for
  $|z| \leq r\delta$.
\end{proof}

\begin{proof}[Proof of \Cref{thm:lipschitz-main}] 

  We set $M = \L_1 \cup \L_2$.
  We are given $f\in \smoo(M)$ and $\epsilon \in \smoo(M, \R_{>0})$.
  Let $R = R_o$ and $\delta$ as choosen in \Cref{cor:vanish-lipshitz}.
  Take an increasing sequence of positive reals $\{r_\eta\}$ such that
  $r_{\eta+2} = R r_\eta$. We define continuous functions $\rho_\eta: [0,+\infty) \to [0,1]$
  such that $\rho_\eta(t)$ is $0$ for $t\in[0,r_\eta]$ and $1$ for $t\in [r_{\eta + 1}, +\infty)$.

  Construct an exhaustion of $M$ by compacts
  $$
  K_\eta = \{z\in M : |z| \leq r_\eta R\}
  $$
  for which the corresponding closed ball as in \Cref{cor:vanish-lipshitz} will be
  $$
  B_\eta = B(0,r_\eta \delta) \text{ and } \epsilon_\eta = \min_{t\in K_\eta}{\epsilon(t)}.
  $$

  As $K_0$ is polynomially convex, we can approximation $f$ on $K_0$ by an entire function $h_0$ such that
  $$
    |h_0(z) - f(z)| < \epsilon_0/2 \quad \forall z \in K_0.
  $$

  We want an entire function $h_1$ such that
  $$
    |h_1(z) - \rho_0(|z|)(f(z) - h_0(z) )| < \epsilon_1/2^2 \quad \forall z \in K_1
  $$ and $|h_1|_{B_1} \leq \epsilon_1/2^2$. Such a $h_1$ can be found using
  \Cref{cor:vanish-lipshitz} as  $f_o(z) = \rho_0(|z|)(f(z) - h_0(z) )$ satisfies the conditions
  with $r = r_0$.

  This process can be done inductively to get $h_\eta$ an entire functions
  such that
  $$
    |h_\eta(z) - \rho_\eta(|z|)(f(z) -\sum_{l=1}^{\eta-1} h_l(z) )| < \epsilon_\eta/2^{\eta+1} \quad \forall z \in K_\eta
  $$
  and $|h_\eta|_{B_\eta} < \epsilon_\eta/2^{\eta+1}$.

  Define $h = \sum_{\eta \in \mathbb{N}} h_\eta$. Given any compact set $A\subset \C^n$, we can find $\eta$
  such that $A\subset B_\eta$. So, $h$ converges uniformly on $A$ as
  $|\sum_{j>\eta} h_j|_{A} < \epsilon_\eta/2^{\eta}$.
  This implies $h$ is an entire function.

  We are done if we show $|h(z)-f(z)|< \epsilon(z)$ for $z\in M$.
  Take any $z\in M$,
  say $r_{\eta_o} < |z| \leq r_{\eta_o +1}$.
  $$
    |h(z) - f(z)|= 
    \left|\sum_{\eta =0}^{\infty} h_\eta(z)  - f(z) \right|
     \leq \sum_{\eta > \eta_o}^{\infty} |h_\eta(z)| + \left|\sum_{\eta =0}^{\eta_o} h_\eta  - f(z) \right|
  $$
  \begin{eqnarray*}
     &\leq&  \epsilon/2^{\eta_o+1} + \left| h_{\eta_o}(z) - \rho_{\eta_o}(z) \left( f(z) - \sum_{\eta=0}^{\eta_o-1}h_\eta(z)\right)\right|
     + (1-\rho_{\eta_o}(z))\left|\sum_{\eta =0}^{\eta_o - 1 } h_\eta(z)  - f(z) \right| \\
     &\leq&  \epsilon(z)/2^{\eta_o+1} + \epsilon(z)/2^{\eta_o+1}
     + (1-\rho_{\eta_o}(z))\left|h_{\eta_o -1} - \rho_{\eta_o -1}(z)\left(f(z) - \sum_{\eta =0}^{\eta_o - 2 } h_\eta(z)   \right) \right| \\
     &\leq&  \epsilon(z)/2^{\eta_o+1} + \epsilon(z)/2^{\eta_o+1} + \epsilon(z)/2^{\eta_o} \leq \epsilon(z)
  \end{eqnarray*}

\end{proof}

\section{Carleman approximation on certain surfaces with hyperbolic points}
\label{section:surfaces-hyperbolic-point}

In this section, we will study a family of surfaces $M \subset \C^2$ with a hyperbolic complex point 
that allows Carleman approximation.
We begin by showing Carleman sets behave well under proper holomorphic maps.

\subsection{Proper maps and Carleman sets}

\begin{proposition}
\label{prop:bdd-E-hull-proper-map}
Let $\Phi: \C^n \to \C^n$ be a proper holomorphic map and $M \subset \C^n$ be a closed subset.
If $\Phi^{-1}(M)$ has bounded E-hulls then $M$ has bounded E-hulls.
\end{proposition}

Before we prove the proposition, we will prove
a lemma on preimages and polynomial hulls under proper holomorphic maps.

\begin{lemma}
    Let $\Phi: \C^n \to \C^n$ be a proper holomorphic map.
    If $S \subset \C^n$ is a closed set,
    then $\Phi^{-1}(\widehat S) \subset \widehat{\Phi^{-1}(S)}$.
\end{lemma}

\begin{proof}
Let $S \subset \C^n$ be a closed set. 
Then its hull $\widehat S = \cup \widehat S_j$
where $\{S_j\}$ is some exhaustion of $S$ by compact sets.
It suffices to show 
$\Phi^{-1}(\widehat S_j) \subset \widehat{\Phi^{-1}(S_j)}$.

Let $u: \C^n \to \C^n$ be a plurisubharmonic function.
We can consider the following push forward of $u$:
$$
	v(\eta) = \max_{z\in \Phi^{-1}(\eta)} \{ u(z) \}.
$$
Then $v: \C^n \to \C^n$ is a plurisubharmonic function as shown in  Klimek \cite[Theorem 2.9.26]{klimek1991}.
Note that any plurisubharmonic function from $\C^n$ to $\C^n$ is a surjection.

Fix $j$. Pick $p \in \Phi^{-1}(\widehat S_j)$.
Then $q = \Phi(p) \in \widehat S_j$.
As $\widehat S_j$ corresponds to the plurisubharmonic hull of $S_j$, we have
$v(q) \leq \sup_{\eta \in S_j} v(\eta)$. Hence,
$$
u(p) \leq v(q) \leq \sup_{\eta \in S_j} v(\eta) \leq \sup_{z\in \Phi^{-1}(S_j)} u(z).
$$
Since this holds for any plurisubharmonic $u$, we have $p \in \widehat { \Phi^{-1}(S_j)}$.
    
\end{proof}

\begin{proof}[Proof of \Cref{prop:bdd-E-hull-proper-map}]
Pick compact $K \subset \C^n$.
We need to show $\widehat{M\cup K}\setminus{M \cup K}$ is bounded.
It suffices to show $\widehat{M\cup K}\setminus{M}$ is bounded. 

As $\Phi^{-1}(M)$ has bounded E-hulls and $\Phi$ is proper,
$[\Phi^{-1}(M \cup K)]\ehat \setminus \Phi^{-1}(M \cup K)$ is bounded.
It is clear that $[\Phi^{-1}(M \cup K)]\ehat \setminus \Phi^{-1}(M)$
is also bounded.
By the above lemma, 
$\Phi^{-1}(\widehat{M \cup K}) 
\subset 
[\Phi^{-1}(M \cup K)]\ehat$.
Hence,
$$
\Phi^{-1}(\widehat{M \cup K} \setminus M) =
\Phi^{-1}(\widehat{M \cup K}) \setminus \Phi^{-1}(M) 
\subset [\Phi^{-1}(M \cup K)] \ehat
\setminus \Phi^{-1}(M).
$$
Since $\Phi$ is a continuous map, 
$\widehat{M \cup K} \setminus M$ is bounded.
\end{proof}

\begin{theorem}
    Let $M\subset \C^n$ be a stratified totally real set
    such that $\Phi^{-1}(M) \subset \C^n$ is a Carleman set
    for some proper holomorphic map $\Phi: \C^n \to \C^n$.
    Then $M$ allows Carleman approximation in $\C^n$.
		\label{thm:carleman-proper-map}
\end{theorem}
\begin{proof}
The Carleman set $\Phi^{-1}(M)$ has bounded E-hulls by 
Magnusson-Wold \cite[Theorem 1.2]{magnusson2016characterization}.
By \Cref{prop:bdd-E-hull-proper-map}, $M$ has bounded E-hulls.

Let $M = \cup M_j$ be a exhaustion of $M$ by compact sets.
For each $j$,
the compact set $\Phi^{-1}(M_j) \subset \Phi^{-1}(M)$, satisfies
$\mathcal P (\Phi^{-1}(M_j))  = \mathcal C (\Phi^{-1}(M_j))$
as $\Phi^{-1}(M)$ allows Carleman approximation.
Hence $\Phi^{-1}(M_j)$ is polynomially convex for each $j$.
As $\Phi$ is a proper holomorphic map, $M_j$ is polynomially convex
\cite[Theorem 1.6.24]{stout2007polynomial}.
Hence $M = \cup  M_j = \cup \widehat M_j = \widehat M$. 

Finally, we apply the characterization given by
Magnusson-Wold \cite[Theorem 1.1]{magnusson2016characterization}
for the stratified totally real set $M$.
As $M$ has bounded E-hulls and is polynomially convex, it
is a Carleman set.
\end{proof}

\subsection{Pseudo Lipschitz graphs}

Recall a Lipschitz graph on $\R^n$ is a set
$$
\Gamma_\psi = \{ x + i\psi(x): x\in \R^n\},
$$
where $\psi: \R^n \to \R^n$ satisfies the Lipschitz condition 
$$
\| \psi(x) - \psi(x') \|  \leq \alpha \| x - x' \|,
$$ for all $x,x' \in \R^n$ where $\alpha \in [0,1)$.

\begin{definition}[Pseudo Lipschitz graph]
	Let $f:\R^n \to \C^n$ be a differentiable function and $\alpha \in [0, \frac{1}{2})$
	such that
	$$
		\| f(x) - f(x') \|_{\C^n}  \leq \alpha \| x - x' \|_{\R^n},
	$$ for all $x,x' \in \R^n$.
	Define the set
	$$
	\Lambda_f = \{ x + f(x): x\in \R^n\}.
	$$
	A pseudo Lipschitz graph is a set that can be mapped
	to $\Lambda_f$ by a $\C$-linear automorphism of $\C^n$
	for some $f$ as shown above.
\end{definition}

\begin{proposition}
	A pseudo Lipschitz graph is a Lipschitz graph.
\end{proposition}
\begin{proof}
	Let $f:\R^n \to \C^n$ be a differentiable function 
	such that satisfies the pseudo Lipschitz condition with $\alpha \in [0, \frac{1}{2})$.

	Define $\lambda:\R^n \to \R^n$ as $\lambda(x) = x + \text{Re} f(x)$.
	The function $\lambda$ is bijective:
	Injectivity is shown by
	\begin{equation}
		\label{eqn:lambda-lipschitz-constant-inverse}
		\| \lambda(x) - \lambda(x') \| \geq
		\| x- x' \| - \| \text{Re}f(x) - \text{Re}f(x') \| \geq
		(1-\alpha) \| x- x' \|
	\end{equation}
	for $x,x' \in \R^n$.
	Let $y\in \R^n$.
	Consider the map $L:\R^n \to \R^n$ given by $L(x) = y - \text{Re} f(x)$.
	This map is a contraction:
	$
	\| L(x) - L(x') \| =
	\| \text{Re}f(x) - \text{Re}f(x') \| \leq \alpha \| x- x' \|
	$.
	Since $\R^n$ is a complete metric space, by Banach fixed-point theorem
	we have $L(a) = a$ for some $a \in \R^n$.
    Hence, $\lambda(a) = y$ and $\lambda$ is surjective.

	Then $\Lambda_f = \Gamma_\psi$ for $\psi(t) = \text{Im} f( \lambda^{-1}(t))$:
	$$
	\{ (x + \text{Re} f(x)) + i \text{Im} f(x) : x\in \R^n\} = 
	\{ t + i \text{Im} f(\lambda^{-1}(t)) : t = \lambda(x) \in \R^n\}.
	$$
	Finally, $\psi$ is Lipschitz:
	\begin{eqnarray*}
		\|\psi(t) - \psi(t')\| &=& \| \text{Im} f(\lambda^{-1}(t)) - \text{Im} f(\lambda^{-1}(t')) \| \\
		&\leq& \alpha \| \lambda^{-1}(t)- \lambda^{-1}(t') \| \\
		&\leq& \frac{\alpha}{1-\alpha} \| t - t' \|.
	\end{eqnarray*}
	The last inequality uses \Cref{eqn:lambda-lipschitz-constant-inverse}.
	Since $\alpha \in [0,\frac{1}{2})$, the Lipschitz constant
	$$
	\frac{\alpha}{1-\alpha} < 1.
	$$
\end{proof}

\begin{remark}
We can choose the Lipschitz constant of $\psi$ to be arbitrary small
by choosing the Lipschitz constant of $f$ to be small:
$$
\alpha_\psi =  \frac{\alpha_f}{1-\alpha_f} \leq 2 \alpha_f
$$
	\label{rmk:lipschitz-const-from-pseudo-lipschitz}
\end{remark}

\subsection{Carleman approximation on certain surfaces}

Forstneri\v c and Stout \cite{forstneric-stout1991} have shown local polynomial convexity 
of a surface $M \subset \C^n$ at a hyperbolic point.
A quantitative version of this is shown recently in \cite[Theorem 1.3]{alagandala2025}:
Consider $M \subset \C^2$ in Bishop's normal form, i.e., $M$ is given by the following equation near $0$:
$$
    z_2 = z_1 \bar z_1 + \gamma (z_1^2 + \bar z_1^2) + F(z_1).
$$
Suppose
\begin{equation}
\|F\|_{\mathcal C^2(\overline{ B_r(0))}} \leq \frac{(2\gamma -1)^3}{2^{14}\gamma^3}
	\label{eqn:pull-back-c2-condition}
\end{equation}
for $r\in (0,R]$,
then $M_r := M \cap (\overline{B_r(0)} \times \C)$ is polynomially convex and satisfies
$\mathcal P (M_r) = \mathcal C (M_r)$.
The proof show that preimage of $M_r$ under the proper holomorphic map $\Psi: \C^2 \to \C^2$ given by
$$
\Psi(z_1, z_2) = (z_1, z_1 z_2 + \gamma( z_1^2 + z_2^2))
$$
is the union of two pseudo Lipschitz graphs:
\begin{eqnarray}
    \Lambda_1 &=& \{ (\zeta, \bar \zeta + f(\zeta)): \zeta \in B_r(0) \}, \nonumber \\
    \Lambda_2 &=& \left\{ \left(\zeta, - \frac{1}{\gamma}\zeta - \bar \zeta + g(\zeta)\right): \zeta \in B_r(0) \right\},
    \label{eqn:hyper-two-lipschitz-graphs}
\end{eqnarray}
where $f: \overline{B_r(0)} \to \C$ and $g: \overline{B_r(0)} \to \C$  are $\mathcal C^1$-smooth
functions with $df(0) = dg(0) = 0$.
Further, it can shown by the inequalities in the proof of \cite[Theorem 1.3]{alagandala2025}
and by \cite[Lemma 2.13]{alagandala2025}
that the Lipschitz constants for $f$ and $g$ are bounded above:
\begin{equation}
\label{eqn:hyp-surface-lipschitz-constant-upper-bound}
\alpha_f, \alpha_g \leq \frac{2^8\gamma}{(2\gamma -1)^2} \|F\|_{\mathcal C^2(\overline{ B_r(0))}},
\end{equation}
where
$$
\| f(\eta) -  f(\eta') \| \leq \alpha_f \| \eta - \eta' \|, \quad
\| g(\eta) -  g(\eta') \| \leq \alpha_g \| \eta - \eta' \|.
$$

Now suppose \Cref{eqn:pull-back-c2-condition} holds for all $r > 0$, i.e.,
\begin{eqnarray}
	\|F\|_{\mathcal C^2(\C)} \leq \frac{(2\gamma -1)^3}{2^{14}\gamma^3}
	\label{eqn:hyperbolic-small-requirements-1-pull-back}
\end{eqnarray}

Then we can pull back the whole surface as the union of two pseudo Lipschitz graphs.
After the $\C$-linear change of coordinates given by
$$
(z_1,z_2) \mapsto  
\left(
\frac{2\gamma}{2\gamma +1}  \frac{z_2-z_1}{2i},
\frac{2\gamma}{2\gamma -1}  \frac{z_1+z_2}{2}
\right),
$$
we can show $\Psi^{-1}(M)$ is the union of:
\begin{eqnarray}
	\tilde \Lambda_1 &=& \left\{ x + \tilde f(x) : x \in \R^2 \right\}, \nonumber \\
	\tilde \Lambda_2 &=& (A+i) \left\{ x + \tilde g(x) : x \in \R^2 \right\},
	\label{eqn:pull-back-lipschitz-global-graphs}
\end{eqnarray}
where 
$$
\tilde f(x) = \left( \frac{2\gamma}{2\gamma+1} \frac{f(\zeta)}{2i}, \frac{2\gamma}{2\gamma-1}  \frac{i f(\zeta)}{2}  \right), \quad
\tilde g(x) = \left(  \frac{-g(\eta)}{2},  \frac{- i g(\eta)}{2}  \right)
$$
for $\eta = x_1 + i x_2$, $\zeta = \frac{2\gamma -1}{2\gamma} x_1 + \frac{2\gamma +1}{2\gamma} x_2$ and
$$
A = 
\begin{bmatrix}0 & - \frac{1}{2 \gamma + 1}\\- \frac{1}{2 \gamma - 1} & 0\end{bmatrix}.
$$
The Lipschitz constant of $\tilde f: \R^2 \to \C^2$ and $\tilde g: \R^2 \to \C^2$ are bounded by the Lipschitz constant of $f: \C \to \C$ and $g: \C \to \C$ respectively.
Further, this is dependent on $\|F\|_{\mathcal C^2(\C)}$ by $\Cref{eqn:hyp-surface-lipschitz-constant-upper-bound}$.
Hence,
$$
\quad \alpha_{\tilde f} \leq 2 \frac{2\gamma +1}{2\gamma -1} \alpha_f
\leq  2 \frac{2\gamma +1}{2\gamma -1}  \frac{2^8\gamma}{(2\gamma -1)^2}  \|F\|_{\mathcal C^2(\C)}
$$
$$
\alpha_{\tilde g} \leq \alpha_g \leq  \frac{2^8\gamma}{(2\gamma -1)^2}  \|F\|_{\mathcal C^2(\C)},
$$
By \Cref{rmk:lipschitz-const-from-pseudo-lipschitz}, the Lipschitz graphs $\Gamma_{\psi_1}$ and $\Gamma_{\psi_2}$
corresponding to $\tilde \Lambda_1$ and $\tilde \Lambda_2$. They have Lipschitz constants bounded above: 
\begin{eqnarray*}
	\alpha_{\psi_1} &\leq& 2 \alpha_f 
		\leq  \frac{2^{10}\gamma (2\gamma +1)}{(2\gamma -1)^3}  \|F\|_{\mathcal C^2(\C)}, \\
	\alpha_{\psi_2} &\leq& 2 \alpha_g 
		\leq  \frac{2^9\gamma}{(2\gamma -1)^2}  \|F\|_{\mathcal C^2(\C)}.
\end{eqnarray*}

The union of Lipschitz graphs $\Gamma_{\psi_1}$ and $\Gamma_{\psi_2}$ is polynomially convex:
Given any compact subset $K$ of this union, \cite[Theorem 1.3]{alagandala2025} guarantees the polynomial convexity of $K$.

The following upper bound on the Lipschitz constant ensures that the cones in \Cref{eqn:plane-angular} intersect trivially:
$$\alpha_{\psi_1}, \alpha_{\psi_2} < (2\gamma -1) \left( \frac{2\gamma}{\sqrt{(2\gamma)^2 - 1}} - 1\right).$$
This can be ensured by taking:

\begin{equation}
 \|F\|_{\mathcal C^2(\C)} \leq
	\left( \frac{2\gamma}{\sqrt{(2\gamma)^2 - 1}} - 1\right)
		\frac{{(2\gamma -1)^4} }{2^{10}\gamma (2\gamma +1)}.
	\label{eqn:hyperbolic-small-requirements-2-cone}
\end{equation}

Finally, we will enforce $\alpha_{\psi_1}, \alpha_{\psi_2} \leq \frac{1}{3}$ by taking:
\begin{equation}
 \|F\|_{\mathcal C^2(\C)} \leq
		\frac{1}{3} \frac{{(2\gamma -1)^3} }{2^{10}\gamma (2\gamma +1)}.
	\label{eqn:hyperbolic-small-requirements-3-one-third}
\end{equation}

\begin{proof}[Proof of \Cref{thm:hyperbolic-carleman-main}]
	By \cite[Theorem 1.3]{alagandala2025},
	\Cref{eqn:hyperbolic-small-requirements-1-pull-back} ensures $\Phi^{-1}(M)$ is the union
	of two Lipschitz graphs and is polynomially convex.
	This union of Lipschitz graphs allows Carleman approximation by \Cref{thm:lipschitz-main}
	if we satisfy \Cref{eqn:hyperbolic-small-requirements-2-cone} and \Cref{eqn:hyperbolic-small-requirements-3-one-third}.

	The surface $M$ is a stratified totally real set as it is totally real everywhere but one point.
	Finally, \Cref{thm:carleman-proper-map} shows that $M$ allows Carleman approximation.
\end{proof}

%
%
%
%

\section{Remarks about $\mathcal C^k$-smooth Carleman approximation}
\label{section:ck-smoothness-remarks}

They key idea in approximating over the union of such totally real sets is 
taming the holomorphic function at the intersection. This boil down to the following
from \Cref{prop:sqeeze}:
Let $\Omega$ Stein. Let $U$ and $V$ open sets such that $U \cup V = \Omega$, $W=U\cap V$.
Given $h^{(t)} \in \holo (W)$ for
all $t>0$ such that $h^{(t)} \to 0 $ as $t \to 0^{+}$ (uniform convergence).
Then there exists families of functions $\alpha^{(t)}$ and $\beta^{(t)}$ such that 
$\alpha^{(t)} \to 0$ and $\beta^{(t)} \to 0$  as $t \to 0^{+}$ (uniform convergence) such
that $\alpha^{(t)} - \beta^{(t)} = h^{(t)}$ on $W$.

To obtain this, we look at the following
linear operator $T: \holo(U) \bigoplus \holo(V) \to \holo(W)$
given by
$$T(\alpha, \beta) = (\alpha - \beta)|_{W}$$
Here the spaces are Frechet spaces with semi-norm defined by supremum norm on compact subsets.
The operator $T$ is surjective by Cousins I problem. By the open mapping theorem for Frechet
spaces we obtained $\alpha^{(t)}$ and $\beta^{(t)}$ that converge to zero.

Now, we can consider $\holo(W), \holo(V)$ and $\holo(U)$ to be Frechet spaces with semi-norm
given by $\smoo^k$ norm on compact subsets. So, if we had $h^{(t)} \to 0$ in $\smoo^k$,
then $\alpha^{(t)} \to 0 $ and $\beta^{(t)} \to 0$ in $\smoo^k$.
This give us $\smoo^k$ convergence in \cref{prop:sqeeze}.

We say a 
a $\smoo^k$ function f 
does
$\smoo^k$ Carleman approximation on
a the union of totally real planes $\cup L_j \subset \C^n$ as in \cref{thm:main}
if for any
$\epsilon \in \smoo(S,\R_{>0})$,
there exists an entire function $h$ such that
$$
|\partial_{(j)}^\alpha h(z) - \partial_{(j)}^\alpha f(z)| < \epsilon(z)
$$
for $z\in L_j$ and $|\alpha| \leq k$
where
$
\partial_{(j)}^l := (\phi_j^{-1})_*\left(\frac{\partial}{\partial x_l}\right)
$ and $\alpha$ represents the multi-index for this operator.

To have $\smoo^k$ Carleman approximation of multiple totally real planes,
we must enforce some restriction to the $\smoo^k$ function due to the following:
If we have a $\smoo^k$ function on the union $\cup L_j$ as in \cref{thm:main}, it must
have directional derivatives (about each plane) confirming with Cauchy-Riemann equation up to order $k$ to be approximated
by holomorphic functions in $\smoo^k$.
This gives us a necessary condition, the existance of a holomorphic polynomial $g$
whose directional derivatives about each plane must match with $h$ at the origin.
That is, for all $|\alpha| \leq k $ 
$$
\partial^\alpha_{(j)}h(0) = \partial^\alpha_{(j)}f(0)  
$$

Under this criteria $\smoo^k$ Carleman approximation has
been shown on the union of two totally real planes in 
\cite[Theorem 2]{manne1994carlemantwo}. The same can be generalized to
the union of multiple totally real planes in \cref{thm:main} and the union of Lipschitz graphs \cref{thm:lipschitz-main}
under the condition that the $\smoo^k$ function $f$ has a corresponding a holomorphic polynomial $g$,
such that partial derivatives till order $k$ at the origin in the directions of the planes
of $g$ match with $f$ .
Details of this argument can be found in \cite[Theorem 2]{manne1994carlemantwo} and \cite[Section 4]{manne1993carlemanone}.
\smallskip

\noindent {\bf Acknowledgements.} Sushil Gorai is partially supported by a ARG MATRICS grant (ANRF/ARGM/2025/002958/MTR) and a Core Research Grant (CRG/2022/003560) of ANRF (erstwhile SERB), Govt. of India.

\bibliographystyle{plain}
	\bibliography{biblio.bib}
\end{document}